\documentclass[12pt, a4paper]{article}

\usepackage{amsmath,amsfonts,amsthm}
\usepackage[T1]{fontenc}
\usepackage{fullpage}
\usepackage{multicol}
\usepackage{graphicx}
\usepackage{caption}
\usepackage{subcaption}

\usepackage{color}

\usepackage[symbol]{footmisc}

\theoremstyle{plain}
\newtheorem{theorem}{Theorem}[section]
\newtheorem{proposition}{Proposition}[section]
\newtheorem{corollary}{Corollary}[section]
\newtheorem{lemma}{Lemma}[section]

\theoremstyle{definition}
\newtheorem{remark}{Remark}[section]

\numberwithin{equation}{section}

\newcommand{\argmax}{\mathop{\mathrm{argmax}}}

\newcommand{\bbP}{\mathbb{P}}

\newcommand{\bbR}{\mathbb{R}}

\begin{document}

\title{Maximum likelihood estimators based on the block maxima method}
\author{ Cl\'ement Dombry\\
Univ. Bourgogne Franche-Comt\'e,\\ Laboratoire de Math\'ematiques de Besan\c con,\\
 UMR CNRS 6623, 16 route de Gray,
25030 Besan\c con cedex, France.\\
 Email: clement.dombry@univ-fcomte.fr\\[.5cm]\mbox{ } and \\[.5cm]
Ana Ferreira$^*$\\
Instituto Superior T\'ecnico,\\ Universidade de Lisboa,\\ Av. Rovisco Pais 1049-001 Lisboa, Portugal.\\ Email: anafh@tecnico.ulisboa.pt}
\maketitle
\abstract{The extreme value index is a fundamental parameter in univariate Extreme Value Theory (EVT). It captures the tail behavior of a distribution and is central in the extrapolation beyond observed data. Among other semi-parametric methods (such as the popular Hill's estimator), the Block Maxima (BM) and Peaks-Over-Threshold (POT) methods are widely used for assessing the extreme value index and related normalizing constants. We provide asymptotic theory for the maximum likelihood estimators (MLE) based on the BM method. Our main result is the asymptotic normality of the MLE with a non-trivial  bias depending on the extreme value index and on the so-called second order parameter. Our approach combines asymptotic expansions of the likelihood process and of the empirical quantile process of block maxima.
The results permit to complete the comparison of most common semi-parametric estimators in EVT (MLE and probability weighted moment estimators based on the POT or BM methods) through their asymptotic variances, biases and optimal mean square errors.}

\vspace{0.5cm} \noindent {\bf Key words}: asymptotic
normality, block maxima method, extreme value index,
maximum likelihood estimator, peaks-over-threshold method, probability weighted moment estimator.


\section{Introduction}

The Block Maxima (BM) method, also known as Annual Maxima, after
Gumbel \cite{G58}  is a fundamental method in Extreme Value Theory and has been
widely used. The method is justified under the Maximum Domain of Attraction (MDA) condition:
for an independent and identically distributed (i.i.d.) sample with distribution function  $F$, 
if the linearly normalized partial maxima converges in distribution, then the limit must be a 
Generalized Extreme Value (GEV) distribution. 
In practice, one rarely exactly knows $F$ but the MDA condition holds for most common 
continuous distributions. 

In the BM method, the initial sample is divided into blocks of the same size  
and the MDA condition ensures that the block maxima are approximately GEV distributed. 
The method is commonly used in hydrology and other environmental applications or in insurance and finance when analysing extremes - see e.g. the monographs by Embrechts et al. \cite{EKM97}, Coles \cite{C01}, Beirlant et al. \cite{BGST04}, de Haan and Ferreira \cite{dHF06} and references therein.

The GEV is a three parameter distribution, with the usual
location and scale parameters, and the extreme value index being the
main parameter as it characterizes the heaviness of the tail. Several estimation methods have
been proposed, including the classical maximum likelihood (ML) and
probability weighted moments (PWM) estimators (Hosking \textit{et al.} \cite{HWW85}). 
The asymptotic study of these estimators has been established for a sample from the GEV distribution and asymptotic normality holds 
with null bias and explicit variance (Prescott and Walden \cite{PW80}, Hosking \textit{et al.} \cite{HWW85}, 
Smith \cite{S85}, B\"ucher and Segers \cite{BS16b}). The theory is made quite difficult and technical by the fact that the support of 
the GEV is varying with respect to its parameters. Regularity in quadratic mean of the GEV model has been proven only recently by B\"ucher and Segers \cite{BS16b} and we provide here a different and somewhat simpler proof (cf. Proposition~\ref{prop:DQM}).

However, in applications, the sample block maxima are only approximately GEV so that the
classical parametric theory suffers from model misspecification. 
In this paper, we intend to fill this gap for ML estimators (MLE), by showing asymptotic normality under 
a flexible second order condition (a refinement of the MDA condition). 
Depending on the asymptotic block size, a non trivial bias may appear in the limit for which we provide an exact expression. 
Recently Ferreira and de Haan \cite{FdH15} showed asymptotic normality of the PWM estimators under the same
conditions. They derived a uniform expansion for the empirical quantile of block maxima that is a crucial tool in our approach as well.
Indeed, the MLE can be seen as a maximizer of the so-called likelihood process. Expressing the likelihood process in terms of this empirical quantile process, 
we are able to derive an expansion of the likelihood process that implies the asymptotic normality of the MLE. This derivation is again made quite technical by the fact that the support of the GEV is varying. Note that the asymptotic normality for the MLE of a Fr\'echet distribution based on the block maxima of a stationary heavy-tailed time series has been obtained by B\"ucher and Segers \cite{BS16c}. There the issue of parameter dependent supports is avoided but time dependence has to be dealt with. Besides, the ideas underlying their proof are quite different.

The asymptotic normality result in the present paper brings novel results to the theoretical comparison of the main semi-parametric estimation procedures in EVT.
On the one hand it permits to compare BM and Peaks-over-Threshold (POT) methods (see e.g. Balkema and de Haan \cite{BdH74}, Pickands \cite{P75}), the latter being another fundamental method in EVT and concurrent with BM. We discuss and compare the four different approaches -- MLE/PWM estimators in the BM/POT approaches -- based on exact theoretical formulas for asymptotic variances, biases and optimal mean square errors depending on the extreme value index and the second order parameter. It turns out that MLE under BM has minimal asymptotic variance among all combinations MLE/PWM and BM/POT but, on the other hand it has some significant asymptotic bias. When analysing the asymptotic optimal mean square error that balances variance and bias, the most efficient combination turns out to be MLE under POT (e.g. Drees, Ferreira and de Haan 2004). It turns out that the optimal sample size is larger for POT-MLE than for BM-MLE, giving a theoretical justification to the heuristic that  POT allows for a better use of the data than BM.

The outline of the paper is as follows: In Section 2 we present the main theoretical conditions and results including Theorem \ref{theo2} giving the asymptotic normality of the MLE. In Section 3 we present a comparative study of asymptotic variances and biases, optimal asymptotic mean square errors and optimal samples sizes among all combinations MLE/PWM and BM/POT. In Section 4 we state additional theoretical statements, including the local asymptotic normality of MLE under the fully parametric GEV model, and provide all the proofs. Finally, Appendix A gathers some formulas for the information matrix and for the bias of BM-MLE and  Appendix B provides useful bounds for the derivatives of the likelihood function that are necessary for the main proofs.

\section{Asymptotic behaviour of MLE}

\subsection{Framework and notations}
The GEV distribution with index $\gamma$ is defined by
\[
G_\gamma(x)=\exp\left(-(1+\gamma x)^{-1/\gamma}\right),\quad 1+\gamma x>0,
\]
 and the corresponding log-likelihood  by
\begin{equation}\label{eq:defggamma}
g_\gamma(x)=\left\{ \begin{array}{l}-(1+1/\gamma)\log(1+\gamma
x)-(1+\gamma x)^{-1/\gamma}\quad \mbox{if}\ 1+\gamma x>0 \\ -\infty
\quad\mbox{otherwise}.\end{array}\right.
\end{equation}
For $\gamma=0$, the formula is interpreted as  $g_0(x)=- x-\exp(-x)$.
The three parameter model with index  $\gamma$, location $\mu$ and scale $\sigma>0$ is defined by the log-likelihood
\begin{equation}\label{eq:defGEV}
 \ell(\theta,x)=g_\gamma\left(\frac{x-\mu}{\sigma}\right)-\log\sigma,\quad \theta=(\gamma,\mu,\sigma).
\end{equation}

A distribution $F$ is said to belong to the max-domain of attraction of the extreme value distribution $G_{\gamma_0}$, denoted by $F\in D(G_{\gamma_0})$, if there exist normalizing sequences $a_m>0$ and $b_m$
such that
\[
\lim_{m\to+\infty} F^m(a_mx+b_m)=G_{\gamma_0}(x),\quad \mbox{for all } x\in\mathbb{R}.
\]
The main aim of the BM method is to estimate the extreme value index $\gamma$ as well as the normalizing constants $a_m$ and $b_m$. The set-up is the following.
Consider  independent and identically distributed (i.i.d.)  random variables $(X_i)_{i\geq 1}$ with common distribution function $F\in D(G_{\gamma_0})$. Divide
the sequence $(X_i)_{i\geq 1}$  into blocks of length $m\geq 1$ and define the $k$-th block maximum by
\begin{equation}\label{eq:def_bmax}
M_{k,m}=\max_{(k-1)m<i\leq km}X_{i},\quad k\geq 1.
\end{equation}
For each $m\geq 1$, the variables $(M_{k,m})_{k\geq 1}$ are i.i.d.
with distribution function $F^m$ and by the max-domain of attraction
condition
\begin{equation}\label{eq:convBM}
\frac{M_{k,m}-b_m}{a_m}\stackrel{d}\longrightarrow G_{\gamma_0}\quad
\mbox{as}\ m\to +\infty.
\end{equation}
This suggests that the distribution of
$M_{k,m}$ is \emph{approximately} a GEV distribution with parameters
$(\gamma_0,b_m,a_m)$. The method consists in pretending that the
sample follows \emph{exactly} the GEV distribution and in maximizing the
GEV log-likelihood so as to compute the MLE. 
A particular feature of the method is that the model is
clearly misspecified since the GEV distribution appears as the limit
distribution of the block maxima as the block size $m$ tends to
$+\infty$ while in  practice we have to use a finite block size. As seen afterwards,  we quantify the misspecification thanks to the so-called second order condition that implies an asymptotic expansion of the empirical quantile process with a non trivial bias term. When plugging this expansion in the ML equations, we obtain a bias term for the likelihood process as well as for the MLE.

The (misspecified) log-likelihood of the $k$-sample  $(M_{1,m},\ldots,M_{k,m})$ is
\begin{equation}\label{eq:def_llp}
L_{k,m}(\theta)=\sum_{i=1}^k  \ell(\theta,M_{i,m}),\quad \theta=(\gamma,\mu,\sigma)\in \Theta=\bbR\times\bbR\times(0,+\infty).
\end{equation}
We say that an estimator $\widehat\theta_n=(\widehat\gamma_n,\widehat\mu_n,\widehat\sigma_n)$ is a MLE if it solves the score equations
\[
\left\{
 \begin{array}{lll}
  \frac{\partial}{\partial \gamma}L_{k,m}(\gamma,\mu,\sigma)&=&0\\
\frac{\partial}{\partial \mu}L_{k,m}(\gamma,\mu,\sigma)&=&0\\
\frac{\partial}{\partial \sigma}L_{k,m}(\gamma,\mu,\sigma)&=&0,\\
  \end{array}
 \right.
\]
which we write shortly in vectorial notation
\begin{equation}\label{eq:score}
\frac{\partial L_{k,m}}{\partial\theta} (\theta)=0.
\end{equation}
A main purpose of this paper is to study the existence and asymptotic normality of the MLE under the following conditions:
\begin{itemize}
	\item First order condition:
\[
F\in D(G_{\gamma_0})\quad \mbox{with } \gamma_0>-\frac{1}{2}.
\]
Note that the also called first order condition \eqref{eq:convBM} is equivalent
to
\[
 \lim_{m\to\infty} \frac{V(mx)-V(m)}{a_m}=\frac{x^{\gamma_0}-1}{\gamma_0},\quad x>0,
\]
with  $V=(-1/\log F)^{\leftarrow}$. W.l.g., we can take $b_m=V(m)$ in Equation~\ref{eq:convBM}, what we shall assume in the following.
\item Second order condition:  for some positive function $a$ and some positive or negative function $A$ with $\lim_{t\to\infty} A(t)=0$,
\begin{equation}\label{eq:2order}
\lim_{t\to\infty}\frac{\frac{V(tx)-V(t)}{a(t)}-\frac{x^{\gamma_0}-1}{\gamma_0}}{A(t)}=\int_1^x
s^{\gamma_0-1}\int_1^s u^{\rho-1}duds=H_{\gamma_0,\rho}(x),\quad x>0,
\end{equation}
with $\gamma_0>-\frac{1}{2}$. Note that necessarily
$\rho\leq 0$ and $|A|$ is regularly varying with index $\rho$.
\item Asymptotic growth for the number $k$ of blocks and the block size $m$: 
\begin{equation}\label{eq:blocksize}
 k=k_n\to \infty,\quad m=m_n\to\infty\ \mbox{and}\ \sqrt{k}A(m)\to\lambda\in\mathbb{R},\quad \mbox{as } n\to\infty.
\end{equation}
\end{itemize}

\subsection{Main results}\label{sec:main}
Before considering the MLE, we focus on the
asymptotic properties of the likelihood and score processes. For the
purpose of asymptotic we introduce the \emph{local parameter}
$h=(h_1,h_2,h_3)\in\bbR^3$:
\begin{equation}\label{def:gmusigmah}
\left\{
\begin{array}{lll}
h_1&=&\sqrt k\left(\gamma-\gamma_0\right)\\
h_2&=&\sqrt k(\mu-b_m)/a_m\\
h_3&=&\sqrt k\left(\sigma/a_m-1\right) \\
\end{array}
\right.
\quad \Leftrightarrow\quad 
 \left\{
 \begin{array}{lll}
 \gamma&=&\gamma_0+ h_1/\sqrt{k}\\
 \mu&=&b_m+a_mh_2/\sqrt{k}\\
 \sigma&=& a_m(1+h_3/\sqrt{k}). 
 \end{array}
 \right.
\end{equation}
Set $\theta_0=\left(\gamma_0,0,1\right)$. The  \emph{local log-likelihood process} at $\theta_0$ is given by
\begin{eqnarray}
\widetilde L_{k,m}(h)&=&L_{k,m}\left(\gamma_0+ \frac{h_1}{\sqrt{k}},b_m+ a_m \frac{h_2}{\sqrt{k}},  a_m+a_m\frac{h_3}{\sqrt{k}}\right)\nonumber\\
&=&\sum_{i=1}^k  \ell\left(\theta_0+\frac{h}{\sqrt{k}},\frac{M_{i,m}-b_m}{a_m}\right)-k\log(a_m),\label{eq:def_lllp}
\end{eqnarray}
and, the \emph{local score process} by
\begin{eqnarray}
\frac{\partial \widetilde L_{k,m}}{\partial h}(h)&=&\frac{1}{\sqrt{k}}\sum_{i=1}^k \frac{\partial \ell}{\partial \theta} \left(\theta_0+\frac{h}{\sqrt{k}},\frac{M_{i,m}-b_m}{a_m}\right)\nonumber\\
&=&\frac{1}{\sqrt{k}}\frac{\partial L_{k,m}}{\partial
\theta}(\theta).  \label{eq:def_lsp}
\end{eqnarray}
Clearly, the score equation \eqref{eq:score} rewrites in this new variable as
$\frac{\partial \widetilde L_{k,m}}{\partial h}(h)=0$. 

In the following, $Q_{\gamma_0}$ denote  the quantile function of the extreme
value distribution $G_{\gamma_0}$, i.e.
\begin{equation}\label{eq:def_qp}
Q_{\gamma_0}(s)=\frac{(-\log s)^{-\gamma_0}-1}{\gamma_0},\quad
s\in(0,1).
\end{equation}

\begin{proposition}\label{prop:concave}
Assume conditions \eqref{eq:2order} and \eqref{eq:blocksize}. Let $r=r_n\to\infty$ be a sequence of positive numbers verifying, as $n\to\infty$,
\begin{equation}\label{eq:condition_rn}
r_n=O(k_n^\delta)\qquad  \mbox{with}\quad 0<\delta<\min(1/2,\gamma_0+1/2).
\end{equation}
Let $H_n\subset\mathbb{R}^3$ be the ball of center $0$ and radius $r_n$.
Then, uniformly for $h\in H_n$,
\begin{equation}\label{eq:theo1.3}
\frac{\partial^2 \widetilde L_{k,m}}{\partial h\partial h^T} (h)=-I_{\theta_0}+o_P(1)
\end{equation}
with $I_{\theta_0}$ the Fisher information matrix
\begin{equation}\label{eq:info}
I_{\theta_0}=-\int_0^1 \frac{\partial^2 \ell}{\partial \theta \partial\theta^T}\left(\theta_0,Q_{\gamma_0}(s)\right)ds.
\end{equation}
As a consequence,  the local log-likelihood process $\widetilde L_{k,m}$ is strictly concave on $H_n$ with high probability.
\end{proposition}

\begin{remark}
 The conditions in Proposition \ref{prop:concave} are sufficient for consistency of MLE, see Dombry \cite{D15}. In particular $\sqrt k A(m)\to\lambda\in\bbR$ implies $m/\log n\to\infty$, the later required for consistency. When $\gamma_0\geq 0$, condition \eqref{eq:condition_rn} implies that \eqref{eq:theo1.3} holds for $h=o(k^{1/2-\varepsilon})$, $\varepsilon>0$.
\end{remark}

Our main result is the following Theorem establishing the asymptotic behavior of the local likelihood process and from which the existence and asymptotic normality of MLE will be deduced.
\begin{theorem}\label{theo1}
Assume conditions \eqref{eq:2order} and \eqref{eq:blocksize}. Then,
the local likelihood process satisfies uniformly
for $h$ in compact sets
\begin{align}
&\widetilde L_{k,m}(h)=\widetilde L_{k,m}(0)+h^T\widetilde G_{k,m}-\frac{1}{2}h^TI_{\theta_0}h+o_P(1),\label{eq:theo1.1}\\
&\frac{\partial \widetilde L_{k,m}}{\partial h} (h)=\widetilde G_{k,m}-I_{\theta_0}h+o_P(1), \label{eq:theo1.2}
\end{align}
where
\begin{equation}\label{eq:normal}
\widetilde G_{k,m}=\frac{1}{\sqrt{k}}\sum_{i=1}^k \frac{\partial
\ell}{\partial \theta}\left(\theta_0,\frac{M_{i,m}-b_m}{a_m} \right)
\stackrel{d}\longrightarrow \mathcal{N}(\lambda b,I_{\theta_0})
\end{equation}
i.e., is asymptotically Gaussian with variance equal to the information matrix and mean depending on the second order condition \eqref{eq:2order} through
\begin{equation}\label{eq:bias}
b=b(\gamma_0,\rho)=\int_0^1 \frac{\partial^2 \ell}{\partial x\partial\theta}\left(\theta_0,Q_{\gamma_0}(s)\right)H_{\gamma_0,\rho}\left(\frac{1}{-\log s} \right)ds
\end{equation}
and on the asymptotic block size through $\lambda$ from \eqref{eq:blocksize}.
\end{theorem}

\begin{remark}{\rm
Explicit formulas for the Fisher information matrix $I_{\theta_0}$ have been given by Prescott and Walden \cite{PW80} (see also Beirlant \textit{et al.} \cite{BGST04} page 169). The vector $b$ given by the integral representation \eqref{eq:bias} can also be computed explicitly. Formulas are provided in Appendix  \ref{app:A}. }
\end{remark}
\begin{remark}{\rm
Equation \eqref{eq:blocksize} requires that both the number of blocks $k$ and the block size $m$ go to infinity
with a relative rate measured by the second order scaling function $A$ and a parameter $\lambda$. When $\lambda=0$, the bias term disappears in \eqref{eq:normal};
this corresponds to the situation where $m$ grows to infinity very quickly with respect to $k$ so that the block size is large enough
and the GEV approximation \eqref{eq:convBM} is very good.}
\end{remark}

Existence and asymptotic normality of the MLE can be deduced from Theorem~\ref{theo1}, mainly by the argmax theorem. The concavity property stated in
Proposition~\ref{prop:concave} plays an important role in the proof of existence and uniqueness.

\begin{theorem}\label{theo2}
Assume conditions \eqref{eq:2order} and \eqref{eq:blocksize}.
\begin{enumerate}
\item[(a)] There exists a sequence of estimators $\widehat\theta_n=(\widehat\gamma_n,\widehat\mu_n,\widehat\sigma_n)$, $n\geq 1$, such that
 \begin{equation}\label{eq:theo2.1}
 \lim_{n\to+\infty}\mathbb{P}\left[\widehat\theta_n \mbox{ is a MLE }\right]=1
 \end{equation}
and
\begin{equation}\label{eq:theo2.2}
 \sqrt{k}\left(\widehat\gamma_n-\gamma_0, \frac{\widehat\mu_n-b_m}{a_m}, \frac{\widehat\sigma_n}{a_m}-1\right)\stackrel{d}\longrightarrow \mathcal{N}(\lambda I_{\theta_0}^{-1}b,I_{\theta_0}^{-1}).
\end{equation}
\item[(b)] If $\widehat\theta^i_n=(\widehat \gamma^i_n,\widehat \mu^i_n, \widehat \sigma^i_n)$, $i=1,2$,are  two sequences of estimators satisfying
\[
 \lim_{n\to+\infty}\mathbb{P}\left[\widehat\theta_n^i \mbox{ is a MLE }\right]=1
\]
and
\[
 \lim_{n\to+\infty}\mathbb{P}\left[\sqrt{k}\left(\widehat\gamma_n^i-\gamma_0, \frac{\widehat\mu_n^i-b_m}{a_m}, \frac{\widehat\sigma_n^i}{a_m}-1\right)\in H_n\right]=1,
\]
then $\widehat\theta_n^1$ and $\widehat\theta_n^2$ are equal with high probability, i.e.
\[
\lim_{n\to +\infty}\mathbb{P}\left[\widehat\theta_n^1=\widehat\theta_n^2\right]=1.
\]
\end{enumerate}
\end{theorem}

\begin{remark}
An interesting by-product of the strict concavity stated in Proposition~\ref{prop:concave} is the convergence of numerical procedures for the computation of the MLE that are implemented in software. The Newton-Raphson algorithm is commonly used to solve numerically the score equation \eqref{eq:score}. Strict concavity of the objective function on a large neighbourhood of the solution  ensures convergence of the algorithm with high probability as soon as the initial value $\theta=(\gamma,\mu,\sigma)$ belongs to this neighbourhood.
\end{remark}

\section{Theoretical comparisons: BM vs POT and MLE vs PWM}

The POT method uses observations above some high threshold or top order statistic and the underlying approximate model is the Generalized Pareto distribution (Balkema and de Haan \cite{BdH74}, Pickands \cite{P75}). Estimators of the shape parameter $\gamma$, as well as location and scale parameters have been proposed and widely studied, including MLE and PWM (Hosking and Wallis \cite{HW87}). For their asymptotic properties - under basically the same conditions as under BM in Theorem \ref{theo2} - we refer to de Haan and Ferreira \cite{dHF06}. Asymptotic normality of PWM estimators under BM has been established only recently by Ferreira and de Haan \cite{FdH15} and a comparison of PWM estimators under BM and POT has been carried out. The aim of the present section is to include our new asymptotic results for MLE estimators under BM, completing the picture in the comparison of the four different cases BM/POT and MLE/PWM. 

Recall that asymptotic normality of MLE (resp. PWM estimator) holds for $\gamma>-1/2$ (resp. $\gamma<1/2$). The number $k$ of selected observations  corresponds to the number of blocks in BM and of selected top order statistics in POT. Similarly as in Ferreira and de Haan \cite{FdH15}, our comparative study is restricted to the range $\rho\in [-1,0]$ where second order conditions for BM and POT are comparable (cf. Drees et al. \cite{DdHL03} or Ferreira and de Haan \cite{FdH15}). In the following we compare MLE/PWM under BM/POT methods through their: (i) asymptotic variances (VAR), (ii) asymptotic biases (BIAS), (iii) optimal asymptotic mean square errors (AMSE) and optimal number of observations minimizing AMSE ($k_0$).

\subsubsection*{(i) Asymptotic variances}
The asymptotic variance depends on $\gamma$ only and is plotted in Figure \ref{VarBMPOT.fig} where straight lines stand for MLE and dashed lines for PWM estimators. Among all four different cases, BM-MLE is the one with the smallest variance within its range. Moreover, for both estimators, BM has  the lowest variance indicating that BM is preferable to POT when variance is concerned.

\begin{figure}
	\centering
	\includegraphics[width=10cm,height=4cm]{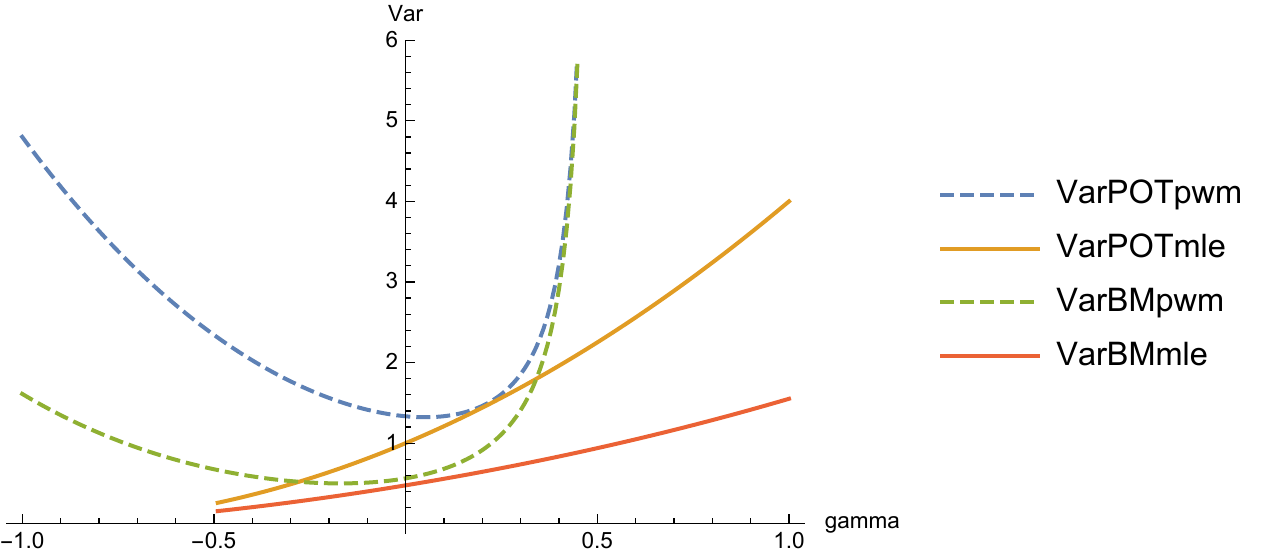}
	\caption{Asymptotic variances of estimators of the extreme value index $\gamma$. The straight lines corresponds to the MLE under BM and POT while the dashed lines correspond to PWM under BM and POT.} \label{VarBMPOT.fig}
\end{figure}

\subsubsection*{(ii) Asymptotic biases}
The asymptotic biases depend on $\gamma$ and $\rho$ and are shown in Figures \ref{bias}--\ref{biasratio}: POT-MLE is the one with the smallest bias also in absolute value when compared to BM-MLE, contrary to what was observed for variance. This is in agreement with what has been observed when comparing BM-PWM and POT-PWM, also shown in Figures \ref{bias}--\ref{biasratio} already analysed in Ferreira and de Haan \cite{FdH15}. There is again the indication that POT method is favourable to BM when concerning bias.

\begin{figure}
	\centering
	\includegraphics[width=10cm,height=4cm]{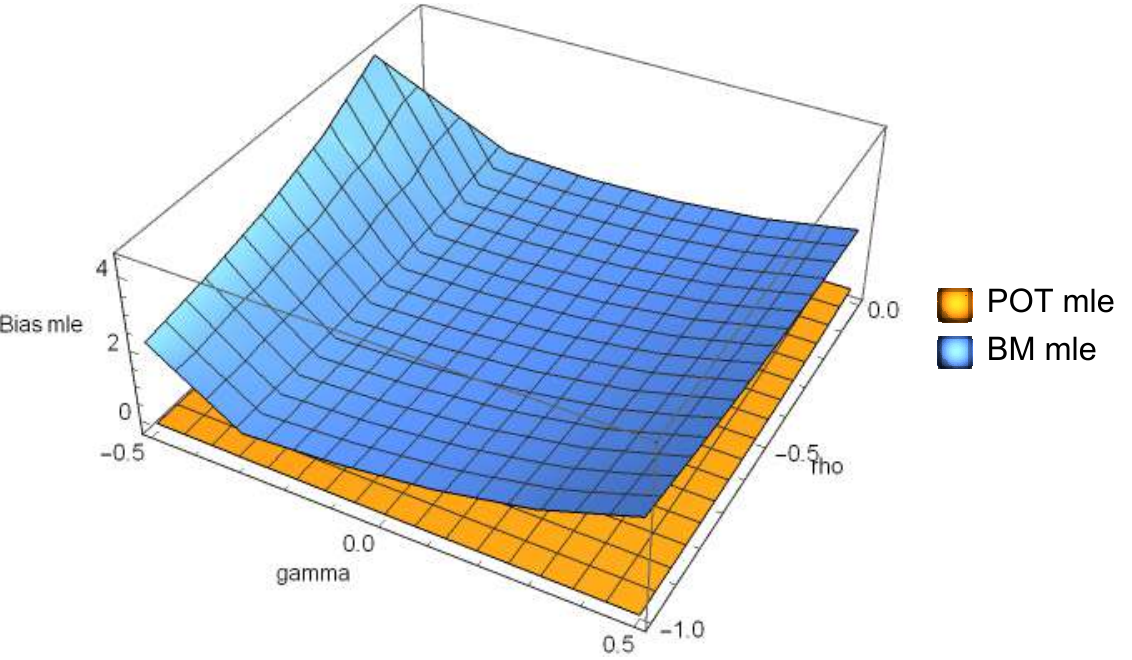}
	\caption{Asymptotic bias of estimators of the extreme value index $\gamma$: blue color for BM and orange for POT.} \label{bias}
\end{figure}

\begin{figure}
	\centering
	\begin{subfigure}{.5\textwidth}
		\centering
		\includegraphics[width=.65\linewidth]{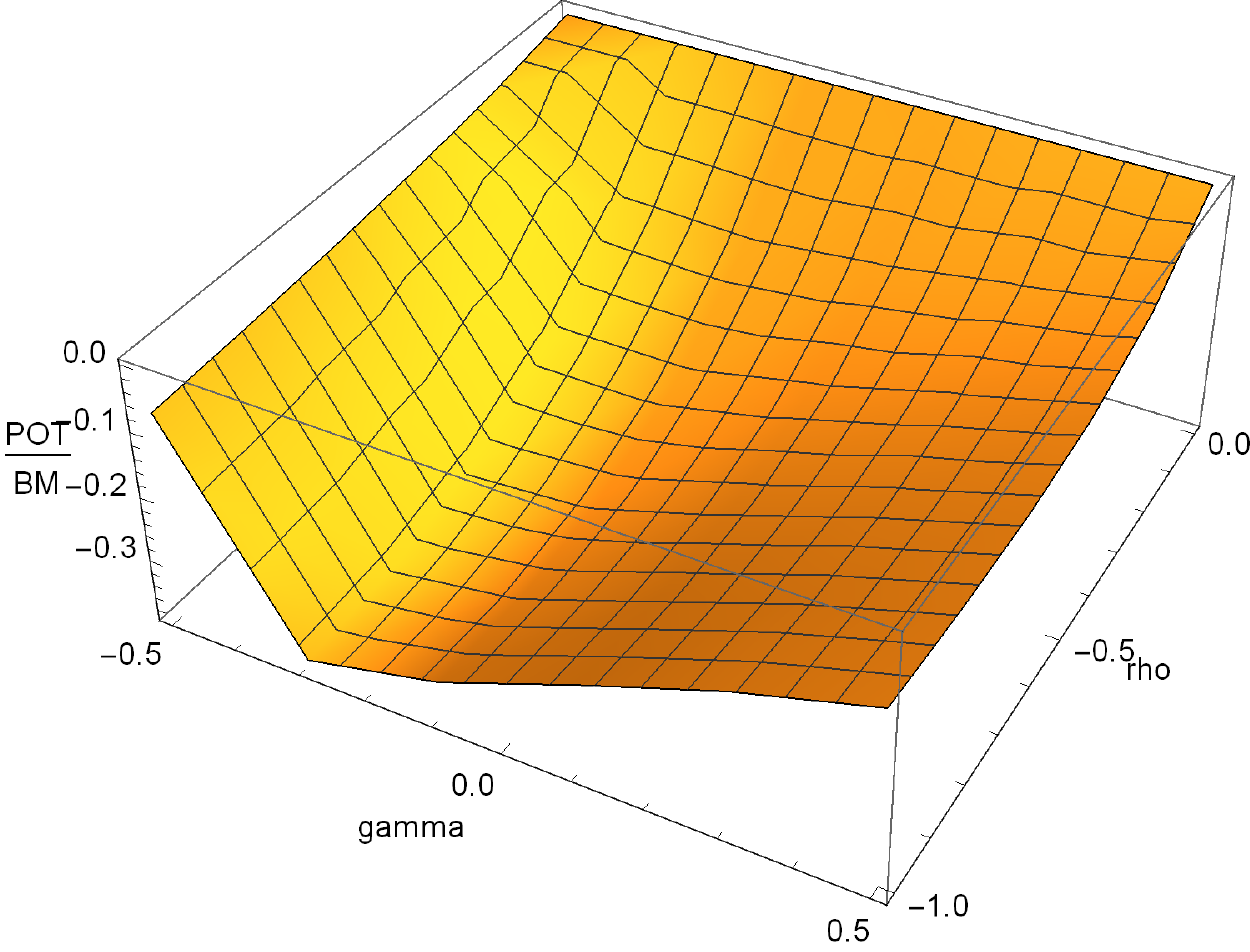}
		\caption{MLE}
		\label{biasBMPOTMLE2}
	\end{subfigure}%
	\begin{subfigure}{.5\textwidth}
		\centering
		\includegraphics[width=.65\linewidth]{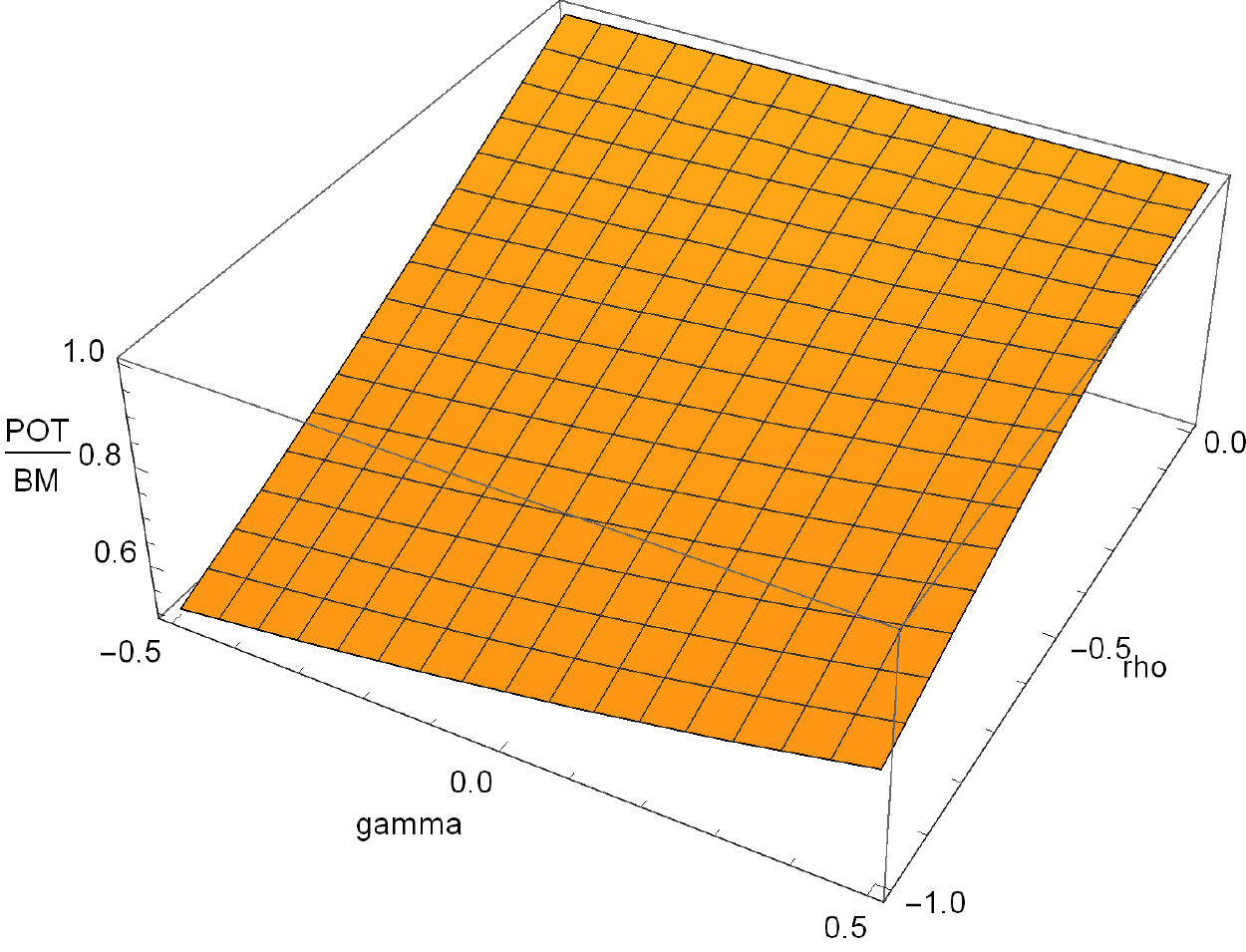}
		\caption{PWM}
		\label{biasMLE2}
	\end{subfigure}
	\caption{Ratios of asymptotic bias: $\text{BIAS}_{POT}/\text{BIAS}_{BM}$.}
	\label{biasratio}
	\end{figure}

\subsubsection*{(iii) Optimal asymptotic MSEs and optimal number of observations}
Another way to compare the estimators that combines both variance and bias information is through mean square error. One can compare these for the optimal number of observations $k_0$ i.e., that value for which the asymptotic mean square error (AMSE) is minimal. Similarly as in Ferreira and de Haan \cite{FdH15}, under the conditions of Theorem \ref{theo2}, we have
\[
k_0\sim\frac n{\left(\frac 1 s\right)^{\leftarrow}(n)} \;\left(\frac{\text{VAR}^2(\gamma)}{\text{BIAS}^2(\gamma,\rho)}\right)^{1/(1-2\rho)}, \qquad n\to\infty, 
\]
with  $s(\cdot)$ a decreasing and $2\rho-1$ regularly varying function such that $A^2(t)=\int_t^\infty s(u)\,du$. It follows in particular that the optimal $k_0$ is different but of the same order for both estimators and methods. As for the optimal AMSE, we have
\[
\text{AMSE}\sim \frac{1-2\rho}{-2\rho}\,\frac {\left( 1/ s\right)^{\leftarrow}(n)}n\,\left(\text{BIAS}^2(\gamma,\rho)\right)^{1/(1-2\rho)}\left(\text{VAR}(\gamma)\right)^{-2\rho/(1-2\rho)} ,\  \quad n\to\infty.
\]
When considering ratios of optimal AMSE (or optimal number $k_0$ of selected observations), the regularly varying function  cancels out and the asymptotic ratio does not depend on $n$ but only on $\gamma$ and $\rho$.

Figure \ref{contour} shows the contour plots of the ratio $\text{AMSE}_{POT}/\text{AMSE}_{BM}$ for MLE and PWM estimators. It is surprising to see a reverse behaviour in both cases: in the range of parameters considered, POT is preferable when MLE are considered, while BM is mostly preferable for PWM estimators. 

\begin{figure}
	\centering
	\begin{subfigure}{.5\textwidth}
		\centering
		\includegraphics[width=.5\linewidth]{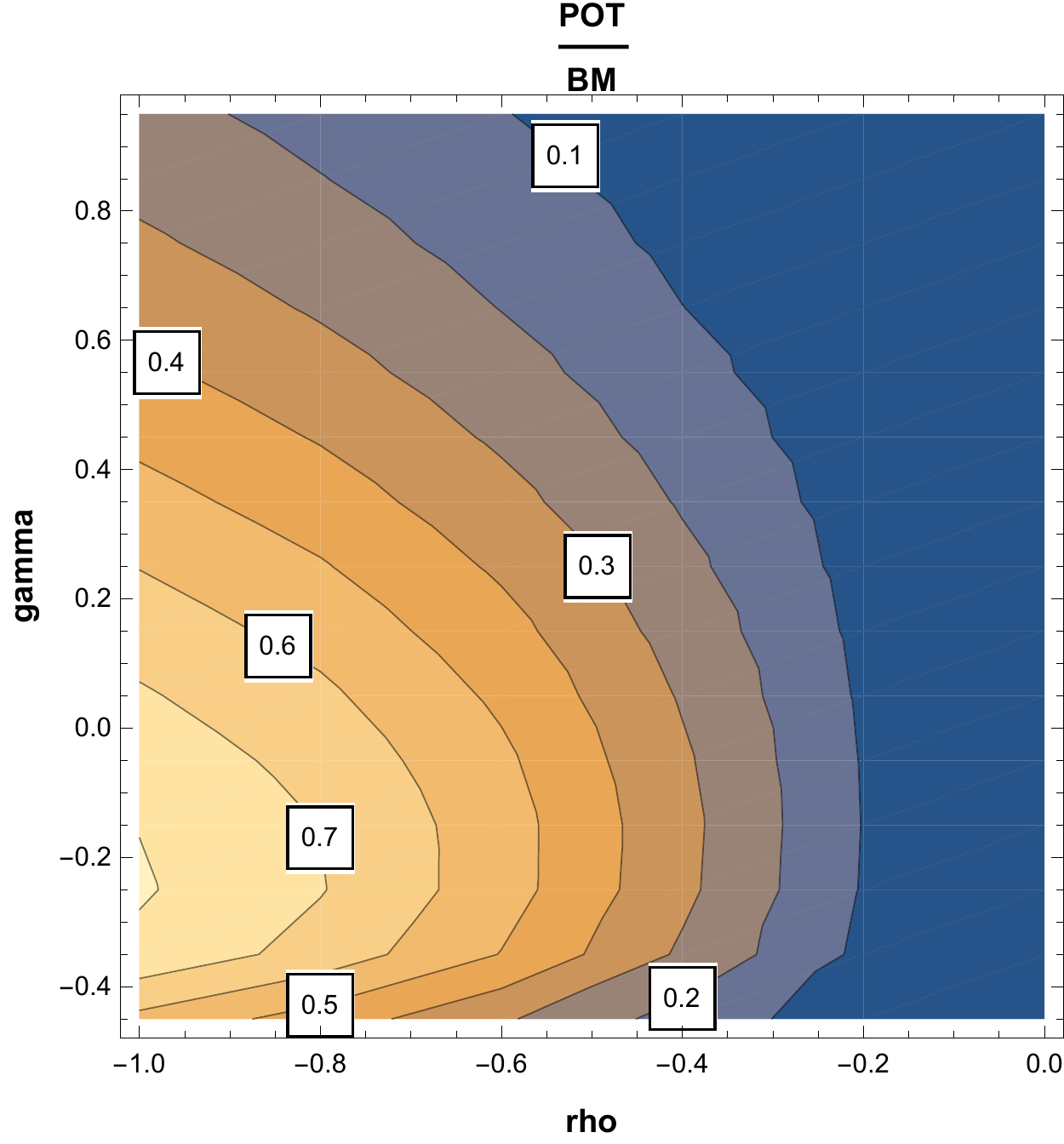}
		\caption{MLE}
		\label{ratios-AMSE-MLE}
	\end{subfigure}%
	\begin{subfigure}{.5\textwidth}
		\centering
		\includegraphics[width=.5\linewidth]{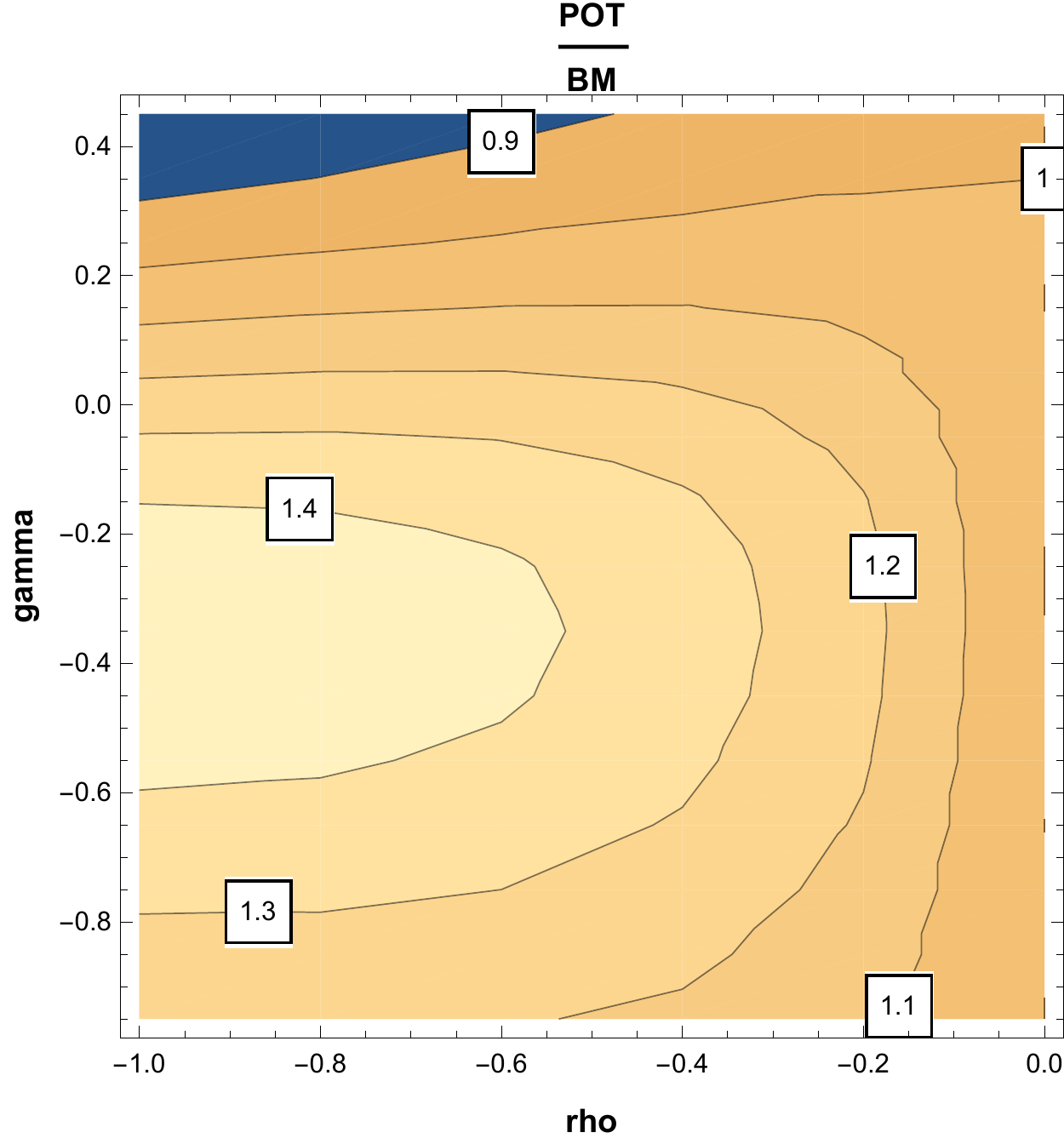}
		\caption{PWM}
		\label{ratios-AMSE-PWM}
	\end{subfigure}
	\caption{Contour plot of ratios of optimal AMSE: $\text{AMSE}_{POT}/\text{AMSE}_{BM}$.}
	\label{contour}
\end{figure}

In Figure \ref{4amse} are shown $ \left(\text{BIAS}^2(\gamma,\rho)\right)^{1/(1-2\rho)}\,\left(\text{VAR}(\gamma)\right)^{-2\rho/(1-2\rho)}$ for comparing optimal AMSE among all combinations. The green surface corresponds to MLE-POT that has always the minimal optimal AMSE  in the range of parameters considered. Finally,  Figure \ref{k0.fig} reports for MLE  the asymptotic ratio of optimal numbers of selected observations, that is $k_{0,POT}/k_{0,BM}$. We can see that the optimal number of observations is larger for POT, which is in agreement with the PWM case considered in previous studies.

\begin{figure}
	\centering
	\begin{minipage}{.5\textwidth}
		\centering
		\includegraphics[width=.75\linewidth]{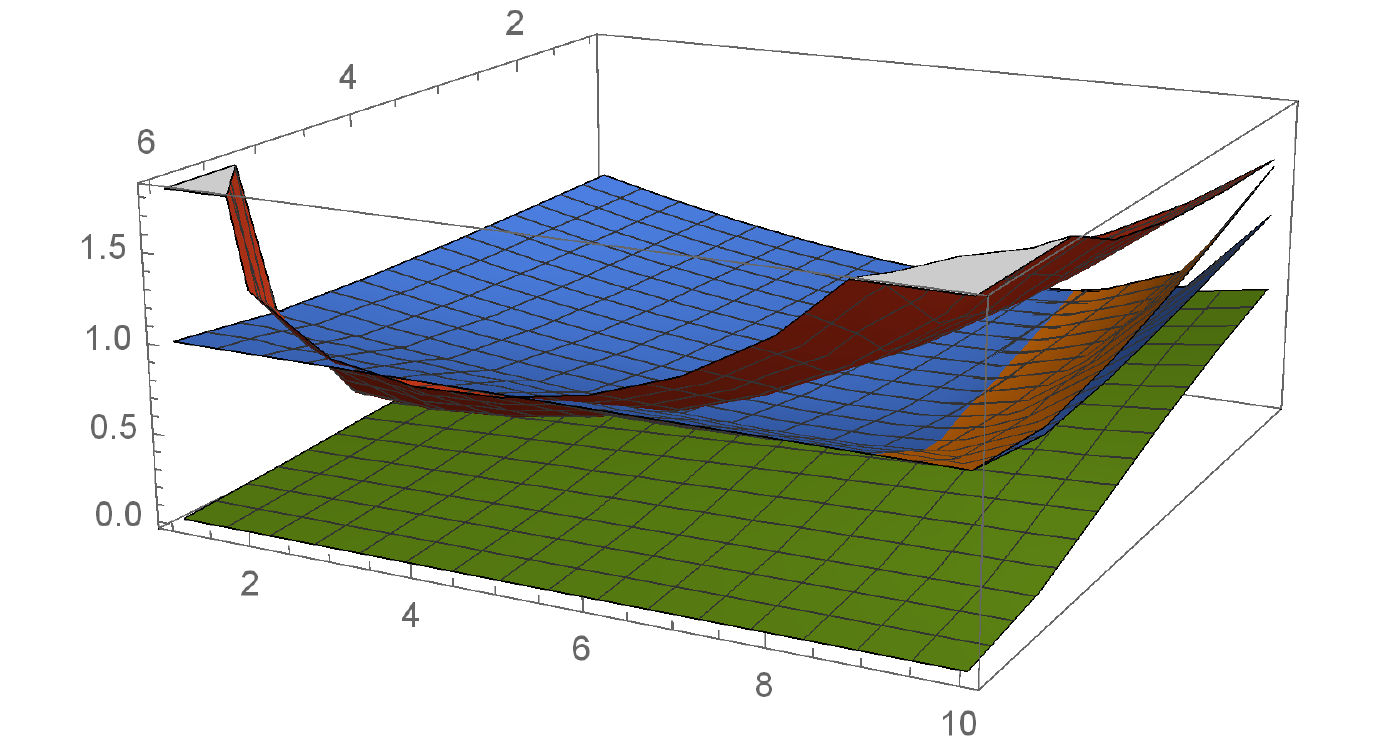}
		\captionof{figure}{Comparison of optimal AMSE: the lowest green surface corresponds to MLE-POT.}
		\label{4amse}
	\end{minipage}%
	\begin{minipage}{.5\textwidth}
		\hspace{-.5cm}
		\includegraphics[width=.5\linewidth]{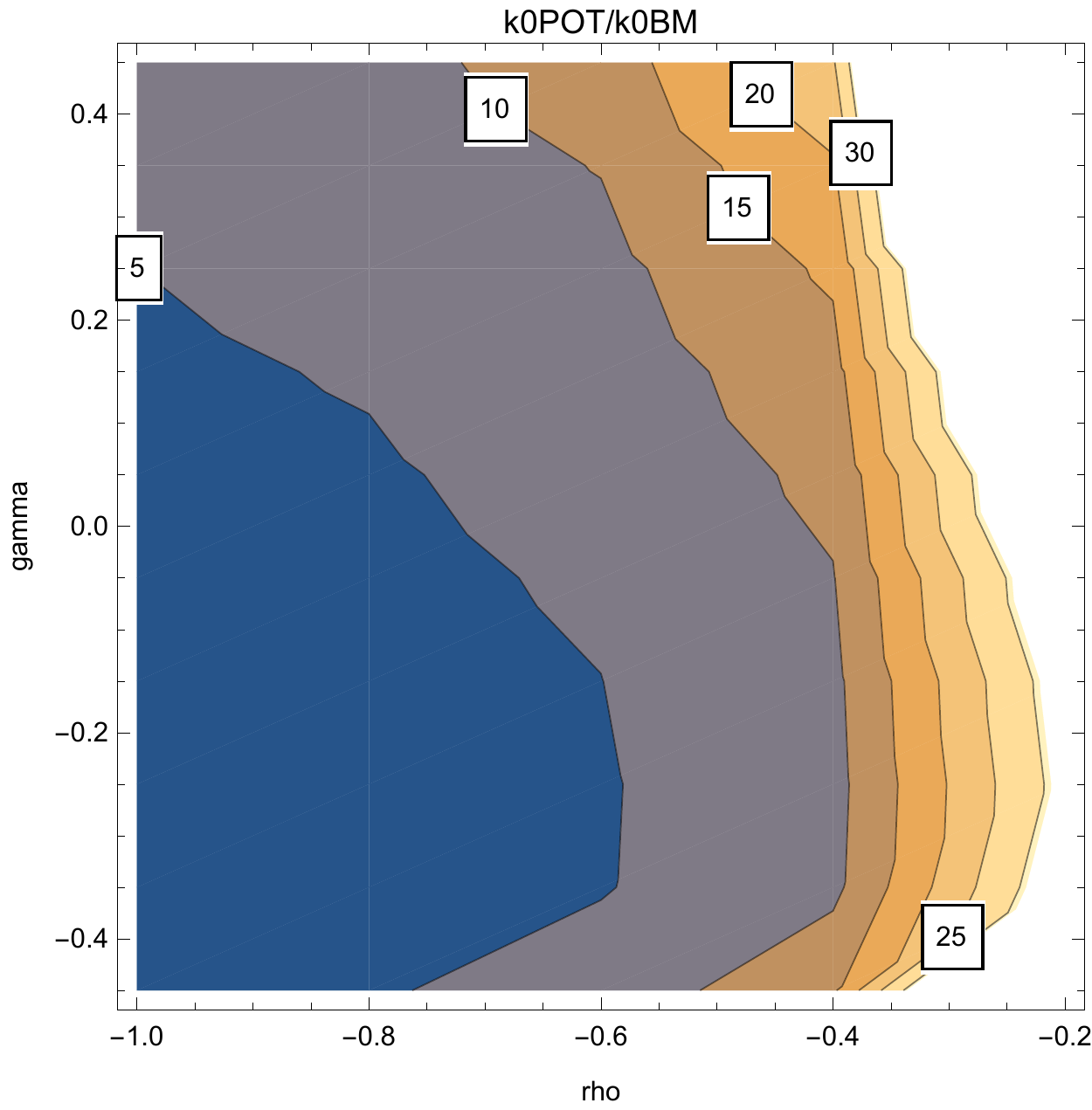}
		\captionof{figure}{Asymptotic ratio of optimal sizes: $k_{0,POT}/k_{0,BM}$.}
		\label{k0.fig}
	\end{minipage}
\end{figure}

\section{Main proofs}

We start by introducing some material that will be useful for the
proofs. More technical material is still postponed to Appendices.

\subsection{Local asymptotic normality of the GEV model}
If the observations $(X_i)_{i\geq 1}$ are exactly $\mathrm{GEV}(\gamma_0,\mu_0,\sigma_0)$ distributed, then the choice of constants
\begin{equation}\label{eq:constants}
 a_m=\sigma_0 m^{\gamma_0}\quad\mbox{and}\quad b_m=\mu_0+\sigma_0 \frac{m^{\gamma_0}-1}{\gamma_0}
\end{equation}
ensures that the normalised block maxima
$\left((M_{i,m}-b_m)/a_m\right)_{k\geq 1}$ are i.i.d. with
distribution $G_{\gamma_0}$. The issue of model misspecification is
irrelevant in that particular case.

In this simple i.i.d. setting, a key property in the theory of
ML estimation is differentiability in quadratic mean
(see e.g. van der Vaart \cite[Chapter 7]{vdV98}).  A
statistical model defined by the family of densities $\{p_\theta(x),
\theta\in\Theta\}$ is called differentiable in quadratic mean at
the point $\theta_0$ if there exists a measurable function $\dot
\ell_{\theta_0}$ called the \textit{score function} such that
\[
 \int_{\mathbb{R}} \left[\sqrt{p_{\theta_0+h}(x)}-\sqrt{p_{\theta_0}(x)}-\frac{1}{2}h^T \dot \ell_{\theta_0}(x)\sqrt{p_{\theta_0}(x)}\right]^2\mathrm{d}x =o(\|h\|^2),\quad \mbox{as} \ h\to 0.
\]
The following Proposition corresponds to Proposition 3.2 in B\"ucher and Segers \cite{BS16b}. We provide a slightly different proof in the case $-1/2<\gamma_0\leq -1/3$.

\begin{proposition}\label{prop:DQM}
The three parameter GEV model with log-likelihood $\ell(\theta,x)$ defined in Equation \eqref{eq:defGEV} is differentiable in quadratic mean at $\theta_0=(\gamma_0,\sigma_0,\mu_0)\in \Theta$ if and only if $\gamma_0>-1/2$. The score function is then given by  $\dot \ell_{\theta_0}(x)=\frac{\partial \ell}{\partial \theta}(\theta_0,x)$.
\end{proposition}
\begin{proof}[Proof of Proposition \ref{prop:DQM}]
The density of the 3-parameter GEV model is given by
\[
p_\theta(x)= \left(1+\gamma\frac{x-\mu}{\sigma}\right)^{-1-1/\gamma}\exp \left(-\left(1+\gamma\frac{x-\mu}{\sigma} \right)^{-1/\gamma}\right)
\]
if $1+\gamma\frac{x-\mu}{\sigma}>0$ and $0$ otherwise. In the case $\gamma_0>-1/3$, the function 
\[
\theta\in (-1/3,+\infty)\times\mathbb{R}\times (0,+\infty)\mapsto \sqrt{p_\theta(x)}
\]
is continuously differentiable for every $x\in\mathbb{R}$ and  the information matrix $\theta\mapsto I_\theta$ is well defined and continuous (see Appendix \ref{app:A} or Beirlant et al. \cite{BGST04} page 169). Differentiability in quadratic mean of the GEV model follows  by a straightforward  application of Lemma 7.6 in Van der Vaart \cite{vdV98}.

In the case $\gamma_0\in(-1/2,-1/3]$, the function $\theta\mapsto \sqrt{p_\theta(x)}$ is not differentiable at points such that $1+\gamma\frac{x-\mu}{\sigma}=0$.
Going back to the definition of differentiability in quadratic mean, we need to show that
\begin{equation}\label{eq:LAN1}
 \lim_{h\to 0}\int_{\mathbb{R}} \left[\frac{ \sqrt{p_{\theta_0+h}(x)}-\sqrt{p_{\theta_0}(x)}-\frac{1}{2}h^T \frac{\partial \ell}{\partial \theta}(\theta_0,x)\sqrt{p_{\theta_0}(x)} }{\|h\|} \right]^2\mathrm{d}x=0.
\end{equation}
This is credible because for all $x\neq \mu-\sigma/\gamma$, the relation
\[
\frac{\partial\sqrt{p_\theta(x)}}{\partial \theta}_{\big| \theta=\theta_0}=\frac{1}{2}\frac{\partial \ell}{\partial \theta}(\theta_0,x)\sqrt{p_{\theta_0}(x)}
\]
 entails
\[
 \lim_{h\to 0} \frac{\sqrt{p_{\theta_0+h}(x)}-\sqrt{p_{\theta_0}(x)}-\frac{1}{2}h^T \frac{\partial \ell}{\partial \theta}(\theta_0,x)\sqrt{p_{\theta_0}(x)}}{\|h\|}=0.
\]
For further reference, we note also that, for $x\neq \mu-\sigma/\gamma$,
\begin{equation}\label{eq:LAN2}
 \frac{\partial^2\sqrt{p_{\theta}(x)}}{\partial \theta\partial \theta^T}_{\big| \theta=\theta_0} = \frac{1}{4}p_{\theta_0}(x)\frac{\partial \ell}{\partial \theta}(\theta_0,x)\frac{\partial \ell}{\partial \theta^T}(\theta_0,x)+\frac{1}{2}\sqrt{p_{\theta_0}(x)} \frac{\partial^2 \ell}{\partial \ell \theta\partial \theta^T}(\theta_0,x).
\end{equation}
A rigorous proof of \eqref{eq:LAN1} is given below. Since $\gamma_0<0$,  we have $\gamma_0+h<0$ for $h=(h_1,h_2,h_3)$ in a neighbourhood of $0$ so that the density $p_{\theta_0+h}$ vanishes outside $(-\infty,x_h)$ with $x_h=(\mu_0+h_2)-(\sigma_0+h_3)/(\gamma_0+h_1)$  the right endpoint of the distribution $\mathrm{GEV}(\gamma_0+h_1,\mu_0+h_2,\sigma_0+h_3)$. We also introduce $\underline{x}_h=\min_{0\leq u\leq 1} x_{uh}$. For all $x<\underline{x}_h$, the function $u\in[0,1]\mapsto \sqrt{p_{\theta_0+uh}(x)}$ is twice continuously differentiable, whence Taylor formula entails
\[
 \sqrt{p_{\theta_0+h}(x)}-\sqrt{p_{\theta_0}(x)}-\frac{1}{2}h^T \frac{\partial \ell}{\partial \theta}(\theta_0,x)\sqrt{p_{\theta_0}(x)}
 =\frac{1}{2}h^T\left(\frac{\partial^2\sqrt{p_{\theta}(x)}}{\partial \theta\partial \theta^T}_{\big| \theta=\theta_0+vh} \right) h
\]
for some $v=v(h,x)\in[0,1]$. Together with Equation \eqref{eq:LAN2}, the formula $(a+b)^2\leq 2(a^2+b^2)$ and Proposition \ref{prop3}, this  yields the upper bound
\begin{align*}
& \left[\frac{ \sqrt{p_{\theta_0+h}(x)}-\sqrt{p_{\theta_0}(x)}-\frac{1}{2}h^T \frac{\partial \ell}{\partial \theta}(\theta_0,x)\sqrt{p_{\theta_0}(x)} }{\|h\|} \right]^2\\
\leq &\frac{1}{32\|h\|^2}\left[p_{\theta_0+vh}(x)^2\left(h^T \frac{\partial \ell}{\partial \theta}(\theta_0+vh,x)\right)^4+4p_{\theta_0+vh}(x) \left( h^T\frac{\partial^2 \ell}{\partial  \theta\partial \theta^T}(\theta_0+vh,x)h\right)^2\right]\\
\leq &C\|h\|^2\Big[p_{\theta_0+vh}(x)^2\max(z(\theta_0+vh,x)^{\gamma_0-\varepsilon},z(\theta_0+vh,x)^{1+\varepsilon})^4 \\
&\hspace{1.5cm}+p_{\theta_0+vh}(x) \max(z(\theta_0+vh,x)^{2\gamma_0-\varepsilon},z(\theta_0+vh,x)^{1+\varepsilon})^2\Big]
\end{align*}
 for all $x<\underline{x}_h$ and $h$ small enough. This entails
\begin{equation}\label{eq:LAN1.1}
 \lim_{h\to 0}\int_{-\infty}^{\underline{x}_h} \left[\frac{ \sqrt{p_{\theta_0+h}(x)}-\sqrt{p_{\theta_0}(x)}-\frac{1}{2}h^T \frac{\partial \ell}{\partial \theta}(\theta_0,x)\sqrt{p_{\theta_0}(x)} }{\|h\|} \right]^2\mathrm{d}x=0.
\end{equation}
It remains to estimate the contribution of the integral between $\underline{x}_h$ and $+\infty$. Recall that $p_{\theta_0+h}(x)$ vanishes for $x\geq x_h$. We have
\begin{align*}
& \frac{1}{\|h\|^2}\int_{\underline{x}_h}^{+\infty}\left[\sqrt{p_{\theta_0}(x)}\right]^2dx=\frac{1}{\|h\|^2}\left[1-G_{\gamma_0}\left(\frac{\underline{x}_h-\mu_0}{\sigma_0}\right)\right],\\
& \frac{1}{\|h\|^2}\int_{\underline{x}_h}^{+\infty}\left[\sqrt{p_{\theta_0+h}(x)}\right]^2dx=\frac{1}{\|h\|^2}\left[1-G_{\gamma_0+h_1}\left(\frac{\underline{x}_h-\mu_0-h_2}{\sigma_0+h_3}\right)\right],\\
& \frac{1}{\|h\|^2}\int_{\underline{x}_h}^{+\infty}\left[h^T \frac{\partial \ell}{\partial \theta}(\theta_0,x)\sqrt{p_{\theta_0}(x)}\right]^2dx\leq \int_{\underline{x}_h}^{x_0} \left\| \frac{\partial \ell}{\partial \theta}(\theta_0,x)\right\|^2p_{\theta_0}(x)dx.
\end{align*}
The first and second integral converge to $0$ as $h\to 0$ because $\underline{x}_h-x_0=O(\|h\|)$ and $\gamma_0>-1/2$. The third integral converges also to $0$ because $\underline{x}_h\to x_0$ and $\gamma_0>-1/2$ so that the score is square integrable (its covariance matrix is $I_{\theta_0}$).
We deduce
\begin{equation}\label{eq:LAN1.2}
 \lim_{h\to 0}\int_{\underline{x}_h}^{+\infty} \left[\frac{ \sqrt{p_{\theta+h}(x)}-\sqrt{p_{\theta}(x)}-\frac{1}{2}h^T \frac{\partial \ell}{\partial \theta}(\theta,x)\sqrt{p_{\theta}(x)} }{\|h\|} \right]^2\mathrm{d}x=0.
\end{equation}
Equations \eqref{eq:LAN1.1} and \eqref{eq:LAN1.2} imply \eqref{eq:LAN1}.

The fact that differentiability in quadratic mean does not hold when $\gamma_0\leq -1/2$ is proved in 
B\"ucher and Segers \cite{BS16b} Appendix C. They observe that for $\gamma_0\leq -1/2$, 
\[
\liminf_{h\to 0} \|h\|^{-2} \int_{\mathbb{R}}^{+\infty}1_{\{p_{\theta_0}(x)=0\}}p_{\theta_0+h}(x)dx>0
\]
which rules out differentiability in quadratic mean. We omit further details here.
\end{proof}

Differentiability in quadratic mean implies that the score function is centered with finite variance equal to the information matrix, i.e.
\begin{equation}\label{eq:scorenul}
 \int_{\mathbb{R}} \dot\ell_{\theta_0}(x)p_{\theta_0}(x)dx=0\quad \mbox{and}\quad
 \int_{\mathbb{R}} \dot\ell_{\theta_0}(x)\dot\ell_{\theta_0}(x)^Tp_{\theta_0}(x)dx=I_{\theta_0}.
\end{equation}
Another important consequence of differentiability in quadratic mean is the local asymptotic normality property of the local score process. The following Corollary follows from Proposition~\ref{prop:DQM} by a direct application of Theorem 7.2 in Van der Vaart \cite{vdV98}.
\begin{corollary}\label{cor:LAN}
Assume that $F=\mathrm{GEV}(\gamma_0,\mu_0,\sigma_0)$  with $\gamma_0>-1/2$ and that the constants $a_m>0$, $b_m\in\bbR$ are given by \eqref{eq:constants}. Then the local log-likelihood process \eqref{eq:def_lllp} satisfies
\[
\widetilde{L}_{k,m}(h)=\widetilde{L}_{k,m}(0)+h^T\widetilde{\widetilde{G}}_{k,m}-\frac{1}{2}h^TI_{\theta_0}h+o_P(1)
\]
where
\[
\widetilde{\widetilde{G}}_{k,m}=\frac{1}{\sqrt k} \sum_{i=1}^k \frac{\partial
\ell}{\partial
\theta}\left(\theta_0,\frac{M_{i,m}-b_m}{a_m}\right)\stackrel{d}\longrightarrow
\mathcal{N}(0,I_{\theta_0}).
\]
\end{corollary}
Note the similarity between Theorem~\ref{theo1} and Corollary~\ref{cor:LAN}.
In Theorem~\ref{theo1} however, the $o_P(1)$ is uniform on compact sets and the model misspecification $F\in \mathcal{D}(G_{\gamma_0})$ results in a bias term $\lambda b$ for the asymptotic distribution $\widetilde G_{k,m}$.

\subsection{The empirical quantile process associated to BM}
The starting point of the proof of Proposition \ref{prop:concave} and Theorem \ref{theo1} is to rewrite the local log-likelihood process \eqref{eq:def_lllp} in terms of the (normalized) empirical quantile process
\begin{equation}\label{eq:def_eqp}
Q_{k,m}(s)=\frac{M_{\lceil ks\rceil:k,m}-b_m}{a_m},\quad 0<s<1,
\end{equation}
where $M_{1:k,m}\leq \cdots\leq M_{k:k,m}$ are  the  order statistics of the block maxima sample $(M_{k,m})_{1\leq k\leq m}$  defined by \eqref{eq:def_bmax} and $\lceil x \rceil $ denotes the smallest integer larger than or equal to $x$. The local log-likelihood process~\eqref{eq:def_lllp} can be rewritten as
\begin{equation}\label{eq:def_lllp2}
\widetilde L_{k,m}(h)=k\int_0^1 \ell\left(\theta_0+\frac{h}{\sqrt{k}},Q_{k,m}(s)\right)ds.
\end{equation}
Convergence \eqref{eq:convBM} ensures the convergence of the empirical quantile process $Q_{k,m}$ to the ``true'' quantile function $Q_{\gamma_0}$ defined in \eqref{eq:def_qp}. The following  expansion of the empirical quantile process is taken from Ferreira and de Haan \cite{FdH15}, Theorem 2.1.
\begin{proposition}\label{prop1}
Assume conditions  \eqref{eq:2order} and \eqref{eq:blocksize}. For a specific choice of the second order auxiliary functions  $a$ and $A$ in~\eqref{eq:2order},
\begin{equation}\label{eq:prop1}
\sqrt{k}\left(Q_{k,m}(s)-Q_{\gamma_0}(s)\right)
= \frac{B_k(s)}{s(-\log s)^{\gamma_0+1}}+\lambda H_{\gamma_0,\rho}\left(\frac{1}{-\log s}\right)
+R_{k,m}(s)
\end{equation}
where $B_k$, $k\geq 1$, denotes an appropriate sequence of standard Brownian bridges and the remainder term $R_{k,m}$ satisfies, for $0<\varepsilon<1/2$,
\begin{equation}\label{eq:remainder}
R_{k,m}(s)= s^{-1/2-\varepsilon}(1-s)^{-1/2-\gamma_0-\rho-\varepsilon}o_P(1)
\end{equation}
uniformly for $s\in  \left[\frac{1}{k+1},\frac{k}{k+1}\right]$.
\end{proposition}
\begin{remark} For Proposition \ref{prop1}, the auxiliary functions $a$ and $A$ have to be specially chosen for establishing uniform second order regular variation bounds refining \eqref{eq:2order}, see Lemma 4.2 in \cite{FdH15}. However, this choice is useful for the proofs only and is irrelevant for the statements of the main results in Section \ref{sec:main}.
\end{remark}

The following Proposition provides useful technical bounds for the proof of the main results.

\begin{proposition}\label{prop2}
    Assume conditions \eqref{eq:2order} and \eqref{eq:blocksize}. Then, as $n\to\infty$,
    \[
    (-\log s)^{\gamma_0}\left(1+(\gamma_0+h_1/\sqrt{k})\frac{Q_{k,m}(s)-h_2/\sqrt{k}}{1+h_3/\sqrt{k}}\right)=e^{O_P(1)}
    \]
    and
    \[
    (-\log s)^{-1}\left(1+(\gamma_0+h_1/\sqrt k)\frac{Q_{k,m}(s)-h_2/\sqrt{k}}{1+h_3/\sqrt{k}}\right)^{-1/(\gamma_0+h_1/\sqrt{k})}=e^{O_P(1)}
    \]
    uniformly for $s\in  \left[\frac{1}{k+1},\frac{k}{k+1}\right]$ and $h\in H_n$ as in Proposition~\ref{prop:concave}.
\end{proposition}

For the proof of Proposition~\ref{prop2}, we need the following Lemma.
\begin{lemma}\label{auxlema} Let $Z_{1:k}<\ldots<Z_{k:k}$ be the order statistics of i.i.d random variables $Z_1,\ldots,Z_k$ with standard Fr\'echet distribution. Then, 
    \[
    \log\left\{(-\log s)Z_{\lceil ks\rceil:k}\right\}=O_P(1)
    \]
    where the $O_P(1)-$term is uniform for $s\in  \left[\frac{1}{k+1},\frac{k}{k+1}\right]$.
\end{lemma}

\begin{proof}[Proof of Lemma~\ref{auxlema}]
    An equivalent statement is, with $U$  standard uniform,
    \[
    \log\left\{\frac{-\log U_{\lceil ks\rceil:k}}{-\log s}\right\}=O_P(1).
    \]
    We use Shorack and Wellner \cite{SW86} (inequality 1 on p.419): for some $M>1$
    \begin{equation}\label{auxlema1}
    \frac 1 M\leq \frac{U_{\lceil ks\rceil:k}}{s}\leq M\qquad \text{ for }\qquad  \frac 1 {k+1} \leq s <1,
    \end{equation}
    \begin{equation}\label{auxlema2}
    \frac 1 M\leq \frac{1-U_{\lceil ks\rceil:k}}{1-s}\leq M\qquad \text{ for }\qquad  0< s\leq \frac k {k+1}.
    \end{equation}
    Relation \eqref{auxlema1} implies, for $s\geq 1/(k+1)$,
    \[
    1-\frac{\log M}{-\log s}\leq \frac{-\log U_{\lceil ks\rceil:k}}{-\log s}\leq 1+ \frac{\log M}{-\log s}.
    \]
    Both sides are bounded for $0<s\leq 1/2$. Relation \eqref{auxlema2} implies, for $s\leq k/(k+1)$,
    \begin{multline*}
    \frac{1-s}{-\log s}\frac{1-U_{\lceil ks\rceil:k}}{1-s}\leq \frac{-\log U_{\lceil ks\rceil:k}}{-\log s} \\
   \leq \frac{-\log\left\{1-\left(1-U_{\lceil ks\rceil:k}\right)\right\}}{-\log s}\leq \frac{1-U_{\lceil ks\rceil:k}}{1-s}\,\frac 1{U_{\lceil ks\rceil:k}}\frac{1-s}{-\log s}.
    \end{multline*}
    Both sides are bounded for $1/2\leq s<1$.
\end{proof}

\begin{proof}[Proof of Proposition \ref{prop2}]
    Let $Z$ be a unit Fr\'echet random variable, i.e. with distribution function $F(x)=e^{-1/x}$, $x>0$, and $\left\{ Z_{i:k}\right\}_{i=1}^k$ be the order statistics
    from the associated i.i.d. sample of size $k$, $Z_1,\ldots,Z_k$. Note that $M_{i:k,m}=^d V(mZ_{i:k})$. From Lemma 4.2 in \cite{FdH15},
    \[
    1+\gamma \frac{V(mZ_{\lceil ks\rceil:k})-b_m}{\sigma}=1+\gamma \frac{a^0_m}{\sigma}\frac{V(mZ_{\lceil ks\rceil:k})-b_m}{a^0_m}
    \]
    is bounded (above and below) by
    \[
    Z_{\lceil ks\rceil:k}^{\gamma_0}+\left(\gamma \frac{a^0_m}{\sigma}-\gamma_0\right)\frac{Z_{\lceil ks\rceil:k}^{\gamma_0}-1}{\gamma_0}+\gamma \frac{a^0_m}{\sigma} A_0(m)H_{\gamma_0,\rho}\left(Z_{\lceil ks\rceil:k}\right)\pm\varepsilon Z_{\lceil ks\rceil:k}^{\gamma_0+\rho\pm\delta}A_0(m)
    \]
    for each $\varepsilon,\delta>0$ provided $k$ and $m$ are large enough. Hence,
    \[
    (-\log s)^{\gamma_0}\left\{1+\gamma \frac{a^0_m}{\sigma}\left(\frac{V(mZ_{\lceil ks\rceil:k})-b_m}{a^0_m}-\frac{\mu-b_m}{a^0_m}\right)\right\}
    \]
    is bounded (above and below) by,
    \begin{multline}
    \left((-\log s)Z_{\lceil ks\rceil:k}\right)^{\gamma_0}+\left(\gamma \frac{a^0_m}{\sigma}-\gamma_0\right)\frac{\left((-\log s)Z_{\lceil ks\rceil:k}\right)^{\gamma_0}-(-\log s)^{\gamma_0}}{\gamma_0}\\
    +\gamma \frac{a^0_m}{\sigma} A_0(m)(-\log s)^{\gamma_0}H_{\gamma_0,\rho}\left(Z_{\lceil ks\rceil:k}\right)-\gamma \frac{a^0_m}{\sigma} (-\log s)^{\gamma_0}\frac{\mu-b_m}{a^0_m}\\
    \pm\varepsilon\left((-\log s)Z_{\lceil ks\rceil:k}\right)^{\gamma_0}Z_{\lceil ks\rceil:k}^{\rho\pm\delta}A_0(m).\label{lemmexpans}
    \end{multline}
    Applying Lemma \ref{auxlema}, the first term in \eqref{lemmexpans} is bounded (above and below) uniformly in $s$, in probability.

    It remains to verify that the other terms are $o_p(1)$ uniformly in $s$.    Note that
    \[
    \sup_{(k+1)^{-1}\leq s\leq k(k+1)^{-1}}\frac{(-\log s)^{\gamma_0}}{k^{1/2-\delta}}=\left\{\begin{array}{ll}
    O\left(\frac{(\log k)^{\gamma_0}}{k^{1/2-\delta}}\right),& \gamma_0>0 \\
    O\left(k^{-1/2-\gamma_0+\delta}\right),& \gamma_0\leq 0 .
    \end{array}\right.
    \]
    Using this with $\delta<\min(1/2,\gamma_0+1/2)$, Lemma \ref{auxlema} and \eqref{eq:condition_rn}, the second term in \eqref{lemmexpans} is $o_p(1)$ uniformly in $s$. For the third term, by
    \[
    A_0(m)(-\log s)^{\gamma_0}H_{\gamma_0,\rho}\left(Z_{\lceil ks\rceil,k}\right)
    =\sqrt k A_0(m)\frac{(-\log s)^{\gamma_0}}{\sqrt k}H_{\gamma_0,\rho}\left(Z_{\lceil ks\rceil,k}\right)
    \]
    and,
    \[
    \sup_{(k+1)^{-1}\leq s\leq k(k+1)^{-1}}\frac{(-\log s)^{\gamma_0}}{\sqrt k}H_{\gamma_0,\rho}\left(\frac1{-\log s}\right)=O\left((\log k)^\xi k^{-1/2-(\gamma_0\vee 0)}\right)
    \]
    for some $\xi\in\bbR$, it follows that it is also $o_p(1)$ uniformly in $s$. The last two terms follow similarly.
    
    For the second statement just note that 
    \[
    (-\log s)^{-\gamma_0/(\gamma_0+h_1/\sqrt{k})}=(-\log s)^{-1}(-\log s)^{h_1/(\gamma_0\sqrt{k})(1+o(1))}
    \]
    where the second factor converges to 1 uniformly in $s$.
\end{proof}

%

The following auxiliary result closely related to Proposition~\ref{prop2}  will be useful in our proofs of Proposition~\ref{prop:concave} and Theorem~\ref{theo1}.

\begin{lemma}\label{lem:boundQ0}
	As $n\to\infty$,
	\[
	(-\log s)^{-1}\left(1+(\gamma_0+h_1/\sqrt{k})\frac{Q_{\gamma_0}(s)-h_2/\sqrt{k}}{1+h_3/\sqrt{k}}\right)^{-1/(\gamma_0+h_1/\sqrt{k})}=1+o(1),
	\]
	uniformly for $s\in\left[\frac{1}{k+1},\frac{k}{k+1}\right]$ and $h\in H_n$ as in Proposition~\ref{prop:concave}.
\end{lemma}
\begin{proof}
 Using \eqref{def:gmusigmah} and expanding,
 	\begin{eqnarray*}
 	&&1+(\gamma_0+h_1/\sqrt{k})\frac{Q_{\gamma_0}(s)-h_2/\sqrt{k}}{1+h_3/\sqrt{k}}\\
 	&=&(-\log s)^{-\gamma_0}+O\left(\max(h_1,h_2,h_3)\frac{1+(-\log s)^{\gamma_0}}{\sqrt k}\right)
 	\end{eqnarray*}
 	where the O-term is in fact a  o-term uniformly in $s$ as seen in the proof of Corollary \ref{prop2}. The result follows as in the last part of the proof of Corollary \ref{prop2} for the second statement.
\end{proof}

\subsection{Proofs of Proposition~\ref{prop:concave}, Theorems~\ref{theo1}~and~\ref{theo2}}

Before proving Proposition \ref{prop:concave}, we first check that, with high probability the local log-likelihood process $\widetilde L_{k,m}$ is finite and twice differentiable on $H_n$.
\begin{lemma}\label{lem:finite}
Under the assumptions of Proposition \ref{theo1}, we have
\begin{equation}\label{eq:finite}
\lim_{n\to \infty}\bbP\left[\widetilde L_{k,m}(h)>-\infty\quad \mbox{for all } h\in H_n\right]=1.
\end{equation}
Furthermore, $\widetilde L_{k,m}$ is smooth on $H_n$ as soon as it is finite on $H_n$.
\end{lemma}
\begin{proof}
In view of Equations \eqref{eq:defggamma},\eqref{eq:defGEV} and \eqref{eq:def_lllp},  $\widetilde L_{k,m}(h)$ is finite on $H_n$ as soon as
\[
1+ (\gamma_0+h_1/\sqrt k)
\frac{Q_{k,m}\left(1/(k+1)\right) -
h_2/\sqrt{k}}{1+h_3/\sqrt{k}}>0\quad \mbox{for all} \ h=(h_1,h_2,h_3)\in H_n.
\]
Proposition \ref{prop2} entails that the left hand side is asymtptically $(\log k)^{-\gamma_0}e^{O_P(1)}$ uniformly on $H_n$ so that it remains positive on $H_n$ with high probability. Equations \eqref{eq:defggamma}-\eqref{eq:defGEV} imply that the function $\theta\mapsto \ell(\theta,x)$ is smooth when it is finite. We deduce that $\widetilde L_{k,m}$ is smooth on $H_n$ as soon as it is finite on $H_n$.
\end{proof}

\begin{proof}[Proof of Proposition~\ref{prop:concave}]
According to Lemma~\ref{lem:finite}, the local log-likelihood process $\widetilde L_{k,m}$ is smooth on $H_n$ with high probability. Differentiating Equation \eqref{eq:def_lllp2}, we get
\[
\frac{\partial^2\widetilde L_{k,m}}{\partial h\partial h^T}(h)=\int_0^1 \frac{\partial^2\ell}{\partial \theta\partial\theta^T}\left(\theta_0+\frac{h}{\sqrt{k}},Q_{k,m}(s)\right)ds.
\]
By the definition \eqref{eq:info} of the information matrix,
\begin{eqnarray*}
\frac{\partial^2\widetilde L_{k,m}}{\partial h\partial h^T}(h)+I_{\theta_0}
&=& \int_0^1 \left(\frac{\partial^2\ell}{\partial \theta\partial\theta^T}\left(\theta_0+\frac{h}{\sqrt{k}},Q_{k,m}(s)\right)- \frac{\partial^2\ell}{\partial \theta\partial\theta^T}\left(\theta_0,Q_{\gamma_0}(s)\right)\right)ds\\
&=& \int_{0}^{\frac{1}{k}}\left(\cdots\right)ds + \int_{\frac{1}{k}}^{\frac{k-1}{k}}\left(\cdots\right)ds+ \int_{\frac{k-1}{k}}^{1}\left(\cdots\right)ds\\
&=& \mathrm{I}+\mathrm{II}+\mathrm{III}
\end{eqnarray*}
We will show that these three terms are $o_P(1)$ uniformly on $H_n$, which proves Equation \eqref{eq:theo1.3}.

\medskip
First consider the boundary terms $\mathrm{I}$ and $\mathrm{III}$. Since $Q_{k,m}(s)$ is constant on $[0,\frac{1}{k}]$, we have
\begin{equation}\label{eq:termI}
\mathrm{I}=\frac{1}{k}\frac{\partial^2\ell}{\partial \theta\partial\theta^T}\left(\theta_0+\frac{h}{\sqrt{k}},Q_{k,m}\left(\frac{1}{k}\right)\right) - \int_0^{\frac{1}{k}} \frac{\partial^2\ell}{\partial \theta\partial\theta^T}\left(\theta_0,Q_{\gamma_0}(s)\right)ds
\end{equation}
The integral term vanishes as $k\to\infty$ because the integral is well defined on $[0,1]$ (see Eq. \eqref{eq:info} or alternatively use the upper bound for the second derivative provided by Proposition~\ref{prop3}). To deal with the first term, we use the upper bound for the second derivative provided by Proposition \ref{prop3}:
\[
 \left\|\frac{\partial^2\ell}{\partial \theta\partial\theta^T}\left(\theta_0+\frac{h}{\sqrt{k}},Q_{k,m}\left(\frac{1}{k}\right)\right) \right\|
\leq C\max\left(z^{-\varepsilon},z^{1+\varepsilon},z^{2\gamma_0-\varepsilon},z^{1+2\gamma_0+\varepsilon}
\right)
\]
with 
\begin{eqnarray*}
z&=&z\left(\theta_0+\frac{h}{\sqrt{k}},Q_{k,m}\left(\frac 1 k\right)\right)\\
&=&\left(1+(\gamma_0+h_1/\sqrt{k})\frac{Q_{k,m}(1/k)-h_2/\sqrt{k}}{1+h_3/\sqrt{k}}\right)^{-1/(\gamma_0+h_1/\sqrt{k})}.
\end{eqnarray*}
Corollary \ref{prop2} with $s=\frac{1}{k}$ provides the asymptotic
$z= e^{O_P(1)}\log k$
uniformly for $h\in H_n$. We deduce that the  first term in \eqref{eq:termI} is asymptotically
\[
e^{O_P(1)}\frac{1}{k}\max\left((\log k)^{-\varepsilon},(\log k)^{1+\varepsilon},(\log k)^{2\gamma_0-\varepsilon},(\log k)^{1+2\gamma_0+\varepsilon}\right)=o_P(1)
\]
uniformly for $h\in H_n$. Hence, $\mathrm{I}=o_P(1)$ uniformly for $h\in H_n$.
The proof for the boundary term $\mathrm{III}$ is similar and details are omitted.

\medskip
Next, consider the main term $\mathrm{II}$. By Taylor formula, we have
\begin{eqnarray*}
\mathrm{II}&=&\int_{\frac{1}{k}}^{\frac{k-1}{k}}
\left( \frac{\partial^2\ell}{\partial \theta\partial\theta^T}\left(\theta_0+\frac{h}{\sqrt{k}},Q_{k,m}(s)\right)- \frac{\partial^2\ell}{\partial \theta\partial\theta^T}\left(\theta_0,Q_{\gamma_0}(s)\right)\right)ds  \\
&=& \mathrm{II}_a+\mathrm{II}_b
\end{eqnarray*}
with
\begin{equation}\label{eq:termIIa}
\mathrm{II}_a=\frac{1}{\sqrt k}\int_{\frac{1}{k}}^{\frac{k-1}{k}}\int_{0}^{1} h^T\frac{\partial^3\ell}{\partial^2 \theta\partial\theta^T}\left(\theta_0+\frac{uh}{\sqrt{k}},(1-u)Q_{\gamma_0}(s)+uQ_{k,m}(s)\right) duds
\end{equation}

\begin{equation}\label{eq:termIIb}
\mathrm{II}_b=\int_{\frac{1}{k}}^{\frac{k-1}{k}}\int_{0}^{1}(Q_{k,m}(s)-Q_{\gamma_0}(s))\frac{\partial^3\ell}{\partial x\partial \theta\partial\theta^T}\left(\theta_0+\frac{uh}{\sqrt{k}},(1-u)Q_{\gamma_0}(s)+uQ_{k,m}(s)\right) duds
\end{equation}
Using the notation $z=z(\theta_n',x_n')$ with
\[
x_n'= (1-u)Q_{\gamma_0}(s)+uQ_{k,m}(s)\quad \mbox{and}\quad
\theta_n'=\theta_0+\frac{uh}{\sqrt{k}},\quad u\in[0,1],
\]
Proposition \ref{prop3} provides the upper bound 
\[
 \left\|\frac{\partial^3\ell}{\partial^2 \theta\partial\theta^T}\left(\theta_n',x_n'\right) \right\|\leq  C\max\left(z^{-\varepsilon}, z^{1+\varepsilon},z^{3\gamma_0-\varepsilon}, z^{1+3\gamma_0+\varepsilon}\right).
\]
Using the fact that  $uh\in H_n$ and that $z(\theta_n',x_n')$ is
between $z(\theta_n',Q_{\gamma_0}(s))$ and $z(\theta_n',Q_{k,m}(s))$,
Proposition~\ref{prop2} and Lemma~\ref{lem:boundQ0} imply
 \[
z=z(\theta_n',x_n')= e^{O_P(1)}(-\log s) 
\]
uniformly for $s\in [\frac{1}{k},\frac{k-1}{k}]$, $u\in [0,1]$ and $h\in H_n$. Using these bounds, we obtain
\[
 \left\|\frac{\partial^3\ell}{\partial^2 \theta\partial\theta^T}\left(\theta_n',x_n'\right) \right\|=e^{O_P(1)}\max\left(s^{-\varepsilon},(1-s)^{-\varepsilon}, (1-s)^{3\gamma_0-\varepsilon}\right)
\]
and, since condition \eqref{eq:condition_rn} implies $\|h\|=O(k^\delta)$,
\[
\mathrm{II}_a=O_P(k^{\delta-1/2})\int_{\frac{1}{k}}^{\frac{k-1}{k}} \max\left(s^{-\varepsilon},(1-s)^{-\varepsilon}, (1-s)^{3\gamma_0-\varepsilon}\right)ds .
\]
When $3\gamma_0-\varepsilon>-1$ the integral converges as $k\to\infty$ and,  $\mathrm{II}_a=O_P(k^{\delta-1/2})=o_P(1)$ since $\delta<1/2$.
When $\gamma_0\leq -1/3$, the integral diverges at rate $O(k^{-1-3\gamma_0+\varepsilon})$ so that $\mathrm{II}_a=O_P(k^{\delta-3/2-3\gamma_0+\varepsilon})=o_P(1)$ since $\delta+\varepsilon<\gamma_0+1/2$.\\
Similarly for the term $\mathrm{II}_b$,  Propositions~\ref{prop2}
and~\ref{prop3} together with Lemma~\ref{lem:boundQ0} imply
\begin{eqnarray*}
 && \left\|\frac{\partial^3\ell}{\partial x\partial \theta\partial\theta^T}\left(\theta_n',x_n'\right) \right\|
\leq  C\max\left(z^{\gamma_0-\varepsilon}, z^{1+\gamma_0+\varepsilon},z^{3\gamma_0-\varepsilon}, z^{1+3\gamma_0+\varepsilon}\right) \\
&\leq& e^{O_P(1)}\max\left((-\log s)^{\gamma_0-\varepsilon},(-\log s)^{1+\gamma_0+\varepsilon},(-\log s)^{3\gamma_0-\varepsilon},(-\log s)^{1+3\gamma_0+\varepsilon}\right),
\end{eqnarray*}
uniformly for $s\in[\frac{1}{k},\frac{k-1}{k}]$, $u\in[0,1]$ and $h\in H_n$. 

From the law of the iterated logarithm,
\begin{equation}\label{eq:lil}
B_k(s)=O_p\left(s^{1/2-\varepsilon}(1-s)^{1/2-\varepsilon}\right) \quad \mbox{uniformly on $(0,1)$}
\end{equation}
and,
\begin{equation}\label{eq:boundH}
H_{\gamma_0,\rho}\left(\frac{1}{-\log s}\right)=O\left(s^{-\varepsilon}(1-s)^{-\varepsilon+\min(-\gamma_0,0)}\right)  \quad \mbox{uniformly on $(0,1)$}.
\end{equation}
Combining these two with Proposition \ref{prop1} it follows,
\begin{equation}\label{eq:boundQ}
\sqrt k \left( Q_{k,m}(s)-Q_{\gamma_0}(s) \right)=O_P\left(s^{-1/2-\varepsilon}(1-s)^{-1/2-\gamma_0-\varepsilon}+s^{-\varepsilon}(1-s)^{\min(-\gamma_0,0)-\varepsilon}\right)
\end{equation}
uniformly on $(0,1)$. 

Combining the previous bound for the derivative and \eqref{eq:boundQ}, we deduce similarly as before
\[
\mathrm{II}_b=O_P\left(\frac{1}{\sqrt{k}}\right)\int_{\frac{1}{k}}^{\frac{k-1}{k}}\max(s^{-1/2-2\varepsilon}, (1-s)^{-1/2+2\gamma_0-2\varepsilon},(1-s)^{3\gamma_0-2\varepsilon})ds=o_P(1)
\]
uniformly in $h\in H_n$ for $\varepsilon$ small enough.
\end{proof}

\begin{proof}[Proof of Theorem~\ref{theo1}]
Integrating Equation \eqref{eq:theo1.3}, we obtain directly
\begin{align}
&\frac{\partial \widetilde L_{k,m}}{\partial h} (h)
= \frac{\partial \widetilde L_{k,m}}{\partial h} (0)-I_{\theta_0}h+o_P(1)\nonumber\\
&\widetilde L_{k,m}(h)
=\widetilde L_{k,m}(0)+h^T  \frac{\partial \widetilde L_{k,m}}{\partial h} (0)-\frac{1}{2}h^TI_{\theta_0}h+o_P(1)\nonumber
\end{align}
uniformly on compact sets. This is exactly Equations \eqref{eq:theo1.1} and \eqref{eq:theo1.2} since
\[
\frac{\partial \widetilde L_{k,m}}{\partial h} (0)= \frac{1}{\sqrt{k}}\sum_{i=1}^k \frac{\partial \ell}{\partial \theta}\left(\theta_0,\frac{M_{i,m}-b_m}{a_m} \right)=\widetilde G_{k,m}.
\]
It remains to prove the asymptotic normality \eqref{eq:normal}, i.e.
\begin{equation}\label{eq:normal2}
\frac{\partial \widetilde L_{k,m}}{\partial h} (0)\to^d \mathcal{N}(\lambda b,I_{\theta_0}).
\end{equation}
Differentiating Equation \eqref{eq:def_lllp2}, we obtain
\[
\frac{\partial \widetilde L_{k,m}}{\partial h} (0)= \sqrt{k}\int_0^1 \frac{\partial \ell}{\partial \theta}\left(\theta_0,Q_{k,m}(s)\right)ds = \mathrm{I}'+\mathrm{II}'+\mathrm{III}'
\]
where the three terms correspond to the integrals on $[0,\frac{1}{k}]$, $[\frac{1}{k},\frac{k-1}{k}]$ and $[\frac{k-1}{k},1]$ respectively. Since $Q_{k,m}(s)$ is constant on the first and last intervals, we have
\begin{align*}
\mathrm{I}'&=\frac{1}{\sqrt{k}}\frac{\partial \ell}{\partial \theta}\left(\theta_0,Q_{k,m}\left(\frac{1}{k+1}\right)\right)\\
\mathrm{III}'&=\frac{1}{\sqrt{k}}\frac{\partial \ell}{\partial \theta}\left(\theta_0,Q_{k,m}\left(\frac{k}{k+1}\right)\right).
\end{align*}
The first term is evaluated thanks to Propositions \ref{prop3} and \ref{prop2}:
\[
\left\|\frac{\partial \ell}{\partial \theta}\left(\theta_0,Q_{k,m}\left(\frac{1}{k+1}\right)\right)\right\| \leq C\max\left( z^{-\varepsilon},z^{1+\varepsilon},z^{\gamma_0-\varepsilon},z^{1+\gamma_0+\varepsilon} \right)
\]
with
\[
z=\left(1+\gamma_0 Q_{k,m}\left(\frac{1}{k+1}\right)\right)^{-1/\gamma_0}=e^{O_P(1)}\log k,
\]
whence we deduce
\[
\mathrm{I}'=e^{O_P(1)}\frac{1}{\sqrt{k}}\max((\log k)^{1+\varepsilon},(\log k)^{1+\gamma_0+\varepsilon}) =o_P(1).
\]
With similar arguments, one can prove $\mathrm{III}'=o_P(1)$.

For the second term, we use Taylor integral formula
\begin{align*}
&\frac{\partial \ell}{\partial \theta}\left(\theta_0,Q_{k,m}(s)\right)
=\frac{\partial \ell}{\partial \theta}\left(\theta_0,Q_{\gamma_0}(s)\right)+\frac{\partial^2 \ell}{\partial x\partial \theta}\left(\theta_0,Q_{\gamma_0}(s)\right)(Q_{k,m}(s)-Q_{\gamma_0}(s))\\
&\quad+(Q_{k,m}(s)-Q_{\gamma_0}(s))^2\int_{0}^1 (1-u)\frac{\partial^3 \ell}{\partial x^2\partial \theta}\left(\theta_0,(1-u)Q_{\gamma_0}(s)+uQ_{k,m}(s)\right)du.
\end{align*}
From decomposition \eqref{eq:prop1} for $\sqrt{k}(Q_{k,m}(s)-Q_{\gamma_0}(s))$, we get
\[
\mathrm{II}'=\sqrt{k}\int_{\frac{1}{k}}^{\frac{k-1}{k}} \frac{\partial \ell}{\partial \theta}\left(\theta_0,Q_{k,m}(s)\right)ds=\mathrm{II}'_a+\mathrm{II}'_b+\mathrm{II}'_c+\mathrm{II}'_d+\mathrm{II}'_e,
\]
with
\begin{align*}
\mathrm{II}'_a&=\sqrt{k}\int_{\frac{1}{k}}^{\frac{k-1}{k}}\frac{\partial \ell}{\partial \theta}\left(\theta_0,Q_{\gamma_0}(s)\right)ds\\
\mathrm{II}'_b&=\int_{\frac{1}{k}}^{\frac{k-1}{k}}\frac{\partial^2 \ell}{\partial x\partial \theta}\left(\theta_0,Q_{\gamma_0}(s)\right)\frac{B_k(s)}{s(-\log s)^{\gamma_0+1}}ds\\
\mathrm{II}'_c&=\lambda\int_{\frac{1}{k}}^{\frac{k-1}{k}} \frac{\partial^2 \ell}{\partial x\partial \theta}\left(\theta_0,Q_{\gamma_0}(s)\right)H_{\gamma_0,\rho}\left(\frac{1}{-\log s}\right)ds\\
\mathrm{II}'_d&=\int_{\frac{1}{k}}^{\frac{k-1}{k}}\frac{\partial^2 \ell}{\partial x\partial \theta}\left(\theta_0,Q_{\gamma_0}(s)\right)R_{k,m}(s)ds\\
\mathrm{II}'_e&=\int_{\frac{1}{k}}^{\frac{k-1}{k}}\int_0^1\sqrt{k}(Q_{k,m}(s)-Q_{\gamma_0}(s))^2(1-u)\frac{\partial^3 \ell}{\partial x^2\partial \theta}\left(\theta_0,(1-u)Q_{\gamma_0}(s)+uQ_{k,m}(s)\right)duds.
\end{align*}
We consider the different terms successively. Equation \eqref{eq:scorenul} implies
\[
\int_{0}^{1}\frac{\partial \ell}{\partial \theta}\left(\theta_0,Q_{\gamma_0}(s)\right)ds=0
\]
so that
\[
\mathrm{II}'_a=-\sqrt{k}\int_{0}^{\frac{1}{k}}\frac{\partial \ell}{\partial \theta}\left(\theta_0,Q_{\gamma_0}(s)\right)ds-\sqrt{k}\int_{1-\frac{1}{k}}^{1}\frac{\partial \ell}{\partial \theta}\left(\theta_0,Q_{\gamma_0}(s)\right)ds.
\]
Proposition~\ref{prop3} provides the upper bound
\begin{eqnarray*}
\left\|\frac{\partial \ell}{\partial \theta}\left(\theta_0,Q_{\gamma_0}(s)\right)\right\|
&\leq& C \max\left((-\log s)^{-\varepsilon},(-\log s)^{1+\varepsilon},(-\log s)^{\gamma_0-\varepsilon},(-\log s)^{1+\gamma_0+\varepsilon}\right)\\
&\leq& C\max(s^{-\varepsilon},(1-s)^{-\varepsilon},(1-s)^{\gamma_0-\varepsilon}),
\end{eqnarray*}
whence we deduce 
$\mathrm{II}'_a=O(\max(k^{-1/2+\varepsilon},k^{-1/2-\gamma_0+\varepsilon}))=o(1)$ because $\gamma_0>-1/2$.
For term $\mathrm{II}'_b$, Proposition~\ref{prop3} entails
\begin{eqnarray}
&&\left\|\frac{\partial^2 \ell}{\partial x\partial \theta}\left(\theta_0,Q_{\gamma_0}(s)\right)\right\|\\
&\leq &C \max((-\log s)^{\gamma_0-\varepsilon},(-\log s)^{1+\gamma_0+\varepsilon},(-\log s)^{2\gamma_0-\varepsilon},(-\log s)^{1+2\gamma_0+\varepsilon})\nonumber \\
&\leq& C\max(s^{-\varepsilon},(1-s)^{\gamma_0-\varepsilon},(1-s)^{2\gamma_0-\varepsilon}).\label{eq:prooftheo1}
\end{eqnarray}
Combined with \eqref{eq:lil} we get
\begin{eqnarray*}
&&\left\|\frac{B_k(s)}{s(-\log s)^{\gamma_0+1}}\frac{\partial^2 \ell}{\partial x\partial \theta}\left(\theta_0,Q_{\gamma_0}(s)\right)\right\|\\
&=& \max(s^{-1/2-2\varepsilon},(1-s)^{-1/2-2\varepsilon},(1-s)^{\gamma_0-1/2-2\varepsilon})O_P(1),
\end{eqnarray*}
which implies, since $\gamma_0>-1/2$,
\[
\mathrm{II}'_b=\int_0^1\frac{\partial^2 \ell}{\partial x\partial \theta}\left(\theta_0,Q_{\gamma_0}(s)\right)\frac{B_k(s)}{s(-\log s)^{\gamma_0+1}}ds+o_P(1)
\]
where the integral on $[0,1]$ is well defined.

Similarly for $\mathrm{II}'_c$, \eqref{eq:boundH} together with \eqref{eq:prooftheo1} yields
\begin{eqnarray*}
&&\left\|H_{\gamma_0,\rho}\left(\frac{1}{-\log s}\right)\frac{\partial^2 \ell}{\partial x\partial \theta}\left(\theta_0,Q_{\gamma_0}(s)\right)\right\|\\
&\leq& C \max\left(s^{-2\varepsilon},(1-s)^{-2\varepsilon},(1-s)^{2\gamma_0-2\varepsilon}\right).
\end{eqnarray*}
Because $2\gamma_0>-1$, we get
\[
\mathrm{II}'_c=\lambda\int_0^1\frac{\partial^2 \ell}{\partial x\partial \theta}\left(\theta_0,Q_{\gamma_0}(s)\right)H_\rho\left(\frac{1}{-\log s}\right)ds+o(1)
\]
where the integral on $[0,1]$ is well defined. 

For $\mathrm{II}'_d$ we use the uniform bound \eqref{eq:remainder} and the upper bound \eqref{eq:prooftheo1} to get
\begin{eqnarray*}
\mathrm{II}'_d&=&\int_{\frac{1}{k}}^{\frac{k-1}{k}}\frac{\partial^2 \ell}{\partial x\partial \theta}\left(\theta_0,Q_{\gamma_0}(s)\right)R_{k,m}(s)ds\\
&=& o_P(1)\int_{\frac{1}{k}}^{\frac{k-1}{k}} C\max\left( s^{-1/2-2\varepsilon},(1-s)^{-1/2-\rho-2\varepsilon},(1-s)^{\gamma_0-1/2-\rho-2\varepsilon} \right)ds\\
&=&o_P(1).
\end{eqnarray*}
We consider finally  the  last term $\mathrm{II}'_e$. 
With the notations $x_n'= (1-u)Q_{\gamma_0}(s)+uQ_{k,m}(s)$ and $z=z(\theta_0,x'_n)$, Proposition \ref{prop3} yields
\[
 \left\|\frac{\partial^3\ell}{\partial^2 x\partial\theta}\left(\theta_0,x_n'\right) \right\|\leq  C\max\left(z^{2\gamma_0-\varepsilon},z^{1+2\gamma_0+\varepsilon}, z^{3\gamma_0-\varepsilon}, z^{1+3\gamma_0+\varepsilon}\right).
\]
Using the fact that  $z(\theta_0,x_n')$ is between
$z(\theta_0,Q_{\gamma_0}(s))$ and $z(\theta_0,Q_{k,m}(s))$,
Corollary \ref{prop3} implies $z(\theta_0,x_n')= e^{O_P(1)}(-\log s)$ so that
\[
 \left\|\frac{\partial^3\ell}{\partial^2 x\partial\theta}\left(\theta_0,x_n'\right) \right\|=O_P(1)\max\left(s^{-\varepsilon},(1-s)^{2\gamma_0-\varepsilon},(1-s)^{3\gamma_0-\varepsilon}\right).
\]
Combining this bound with \eqref{eq:boundQ}, we obtain similarly as before
\begin{eqnarray*}
\mathrm{II}'_e
&=& O\left(\frac{1}{\sqrt k}\right)\int_{\frac{1}{k}}^{\frac{k-1}{k}}\max\left(s^{-1-3\varepsilon},(1-s)^{-1+\gamma_0-3\varepsilon}\right)ds=o_P(1).
\end{eqnarray*}
Collecting all the different terms, we get
\begin{eqnarray*}
\frac{\partial \widetilde L_{k,m}}{\partial h} (0)&=&
\int_0^1\frac{\partial^2 \ell}{\partial x\partial \theta}\left(\theta_0,Q_{\gamma_0}(s)\right)\frac{B_k(s)}{s(-\log s)^{\gamma_0+1}}ds \\
&&\quad +\lambda\int_0^1\frac{\partial^2 \ell}{\partial x\partial \theta}\left(\theta_0,Q_{\gamma_0}(s)\right)H_\rho\left(\frac{1}{-\log s}\right)ds +o_P(1).
\end{eqnarray*}
The second term in the right-hand side is deterministic and corresponds to $\lambda b$ with $b$ defined by \eqref{eq:bias}. The first term is an integral of the Brownian bridge that defines a centered Gaussian vector whose covariance only depends  on the first order parameter $\gamma_0$. Comparing with the special case of i.i.d. GEV random variables considered in Corollary \ref{cor:LAN}, we identify the covariance which is equal to $I_{\gamma_0}$. This proves Equation~\eqref{eq:normal2} and concludes the proof of Theorem~\ref{theo1}.
\end{proof}

\begin{proof}[Proof of Theorem~\ref{theo2}]
The proof of Theorem~\ref{theo2} relies on Theorem~\ref{theo1} and on the Argmax Theorem. Consider the random processes
 \[
 M_n(h)=\widetilde L_{k,m}(h)-\widetilde L_{k,m}(0),\quad h\in\bbR^3
 \]
and
\[
M(h)=h^T(\lambda b+G)-\frac{1}{2}h^TI_{\theta_0}h ,\quad h\in\bbR^3
\]
with $G$ a centered Gaussian random vector with variance $I_{\theta_0}$.
Let $H_n$ be the closed ball of $\bbR^3$ centered at $0$ and with radius $r_n\to\infty$ such that $r_n=O(k^\delta)$ as in \eqref{eq:condition_rn}. Define the maximizer
 \[
 \widehat h_n=\argmax _{h\in H_n} M_n(h).
 \]
In the case it is not unique, define $\widehat h_n$ as the smallest  maximizer in the lexicographic order. Theorem \ref{theo1} implies that, for any compact $K\subset\bbR^3$, $M_n$ converge in distribution  to $M$ in $\ell^\infty(K)$ as $k\to\infty$. The limit process $M$ is continuous and has a unique maximizer given by
\[
\widehat h=\argmax_{h\in\bbR^3} M(h) = I_{\theta_0}^{-1}(\lambda b+G).
\]
The Argmax Theorem (see van der Vaart \cite{vdV98} Corollary 5.58) implies that, provided the sequence $\widehat h_n$ is tight, then $\widehat h_n$ converge weakly to $\widehat h$ as $n\to\infty$.

We now prove the tightness of the sequence $\widehat h_n$. Let $\varepsilon>0$. There is $R>0$ such that
\[
\bbP(\|\widehat h\|< R)>1-\varepsilon.
\]
The relation
\[
M(h)=M(\widehat h)-\frac{1}{2}(h-\widehat h)^TI_{\theta_0}(h-\widehat h),
\]
implies that
\[
\max_{\|\widehat h-h\|\geq 1} M(h)=M(\widehat h)-\frac{1}{2}\lambda_{\min}
\]
with $\lambda_{\min}>0$ the smallest eigenvalue of $I_{\theta_0}$. As a consequence,  $\|\widehat h\|< R$ implies
\[
M(\widehat h)-\max_{\|h\|= R+1}M(h)=\max_{\|h\|\leq R}M(h)-\max_{\|h\|= R+1}M(h)\geq \frac{1}{2}\lambda_{\min}
\]
and this occurs with probability at least $1-\varepsilon$. Using the  convergence in distribution of $M_n$ to $M$ in $\ell^\infty(K)$ with $K=\{h:\|h\|\leq R+1\}$, we deduce that for large $n$
\begin{equation}\label{eq:prooftheo2}
\max_{\|h\|\leq R}M_n(h)-\max_{\|h\|= R+1}M_n(h)\geq \frac{1}{4}\lambda_{\min}
\end{equation}
with probability at least $1-2\varepsilon$. For large $n$,  $H_n$ contains the ball $\{h:\|h\|\leq R+1\}$ (because $r_n\to\infty$) and, according to Proposition~\ref{prop:concave}, $M_n$ is strictly concave on $H_n$ with probability at least $1-\varepsilon$. Then, Equation \eqref{eq:prooftheo2} together with the strict concavity of $M_n$ implies that the maximizer $\widehat h_n$ of $M_n$ over $H_n$ satisfies $\|\widehat h_n\|\leq R$. Hence, for large $n$,
$\bbP(\|\widehat h_n\|\leq R)\geq 1-3\varepsilon$ and this proves the tightness of $\widehat h_n$. Note also that on this event, $\widehat h_n$ belongs to the interior of $H_n$ and is hence a critical point of $M_n$, i.e.
\[
\frac{\partial M_k}{\partial h}(\widehat h_n)=0 \quad \mbox{or equivalently}\quad \frac{\partial \widetilde {L}_{k,m}}{\partial h}(\widehat h_n)=0.
\]
Define
\[
\widehat\theta_n=(\gamma_0+k^{-1/2}\widehat h_{n,1}, b_m+a_mk^{-1/2}\widehat h_{n,2},a_m(1+k^{-1/2}\widehat h_{n,3})).
\]
Equations \eqref{eq:def_lllp} and \eqref{eq:def_lsp} imply that
\[
\widetilde L_{k,m}(\widehat h_n)=L_{k,m}(\widehat \theta_n) \quad \mbox{and}\quad \frac{\partial \widetilde L_{k,m}}{\partial h}(\widehat h_n)=k^{-1/2}\frac{\partial L_{k,m}}{\partial \theta}(\widehat \theta_n).
\]
Hence, with high probability, $L_{k,m}$ has a local maximum at $\widehat\theta_n$ with $\frac{\partial L_{k,m}}{\partial \theta}(\widehat \theta_n)=0$, i.e.  $\widehat\theta_n$ is a MLE  and Eq. \eqref{eq:theo2.1} is satisfied. Eq. \eqref{eq:theo2.2} stating the asymptotic normality of $\widehat\theta_k$ is a direct consequence of
\[
\widehat h_n\stackrel{d}{\rightarrow} \widehat h=I_{\theta_0}^{-1}(\lambda b+G)\sim \mathcal{N}(\lambda I_{\theta_0}^{-1}b,I_{\theta_0}^{-1}).
\]
\medskip
The second part of Theorem~\ref{theo2}, i.e. the asymptotic uniqueness of the MLE, is a consequence of the strict concavity stated in Proposition \ref{theo1}: with large probability the log-likelihood function is strictly concave on $H_n$  and hence the score equation $\frac{\partial}{\partial h}\widetilde{L}_{k,m}(h)=0$ has a unique solution on $H_n$. For $n$ large, the  normalised MLE 
\[
\widehat h^i_n=(\widehat \gamma_n^i,(\widehat\mu_n^i-b_m)/a_m,\widehat\sigma_n^i/a_m-1), \quad i=1,2
\]
belong to $H_n$ with large probability and solve  $\frac{\partial}{\partial h}\widetilde{L}_{k,m}(h)=0$. This implies that $\widehat h^1_n=\widehat h^2_n$ with high probability and hence $\widehat \theta^1_n=\widehat \theta^2_n$ with high probability.
\end{proof}

\appendix

\section{Formulas for the information matrix and bias}\label{app:A}
According to Prescott and Walden  \cite{PW80} (see also Beirlant {\it et al.} \cite[page 169]{BGST04}), the information matrix of the GEV model at point $\theta_0=(\gamma_0,0,1)$ is given by
{\scriptsize
\[
I_{\theta_0}= \left(\begin{array}{ccc}
        \frac{1}{\gamma_0^2} \left(\frac{\pi^2}{6}+\left(1-\gamma_\ast+\frac{1}{\gamma_0}\right)^2-\frac{2q}{\gamma_0}+\frac{p}{\gamma_0^2}\right)& -\frac{1}{\gamma_0}\left(q-\frac{p}{\gamma_0}\right)& -\frac{1}{\gamma_0^2}\left(1-\gamma_\ast-q+\frac{1-r+p}{\gamma_0}\right)\\
       -\frac{1}{\gamma_0}\left(q-\frac{p}{\gamma_0}\right) & p &-\frac{p-r}{\gamma_0} \\
        -\frac{1}{\gamma_0^2}\left(1-\gamma_\ast-q+\frac{1-r+p}{\gamma_0}\right)& -\frac{p-r}{\gamma_0} &\frac{1}{\gamma_0^2}(1-2r+p)
       \end{array}
 \right)
\]
}
where $\Gamma$ is Euler's Gamma function, $\gamma_\ast=0.5772157$ is Euler's constant and
{\scriptsize
\[
 p= (1+\gamma_0)^2\Gamma(1+2\gamma_0),\quad
 q= (1+\gamma_0)\Gamma'(1+\gamma_0)+\left(1+\frac{1}{\gamma_0}\right)\Gamma(2+\gamma_0),\quad
 r= \Gamma(2+\gamma_0).
\]
}

The bias in Theorem \ref{theo2} is given by $I_\theta^{-1}b$ where the vector $b$ can be computed exactly. Calculations are tedious and have been performed with Mathematica$^\circledR$.
We get $b=(b_\gamma,b_\mu,b_\sigma)$ with
{\scriptsize
\begin{eqnarray*}
b_\gamma&=&\int_0^1 \frac{\partial^2 \ell}{\partial x\partial\gamma}\left(\theta_0,Q_{\gamma_0}(s)\right)H_{\gamma_0,\rho}\left(\frac{1}{-\log s} \right)ds\\
&=&\left\{\begin{array}{ll}
\frac{1}{\gamma_0^3\rho(\gamma_0+\rho)}\Big((\gamma_0+\rho)(1+\gamma_0-\gamma_\ast\gamma_0)-(\gamma_0+\gamma_0^2(1+\rho)+2\rho(1+\gamma_0))\Gamma(1+\gamma_0)+(1+\gamma_0)^2\rho\Gamma(1+2\gamma_0)& \\
+\gamma_0^2\Gamma(1-\rho)-\gamma_0(1+\gamma_0)\Gamma(2-\rho)+\gamma_0(1+\gamma_0)(1-\rho)\Gamma(1+\gamma_0-\rho)-\gamma_0\rho\Gamma'(2+\gamma_0)-\gamma_0^2\Gamma'(2-\rho) \Big), &\rho<0,   \\
\frac{1}{\gamma_0^4}\Big((1+\gamma_0-\gamma_0\gamma_\ast)^2+\gamma_0^2\pi^2/6+(1+\gamma_0)^2\Gamma(1+2\gamma_0)-2(1+\gamma_0)\left[(1+\gamma_0)\Gamma(1+\gamma_0)+\gamma_0\Gamma'(1+\gamma_0)\right] \Big), &\rho=0,
\end{array}\right.
\end{eqnarray*}

\begin{eqnarray*}
	b_\mu&=&\int_0^1 \frac{\partial^2 \ell}{\partial x\partial\mu}\left(\theta_0,Q_{\gamma_0}(s)\right)H_{\gamma_0,\rho}\left(\frac{1}{-\log s} \right)ds\\
	&=&\left\{\begin{array}{ll}
		\frac{1+\gamma_0}{\gamma_0\rho(\gamma_0+\rho)}\left(-(\gamma_0+\rho)\Gamma(1+\gamma_0)+(1+\gamma_0)\rho\Gamma(1+2\gamma_0)+\gamma_0(1-\rho)\Gamma(1+\gamma_0-\rho)\right),&\rho<0,   \\
		\frac{(1+\gamma)}{\gamma_0^2}\Big((1+\gamma_0)\Gamma(1+2\gamma_0)-\Gamma(2+\gamma_0)-\gamma_0\Gamma'(1+\gamma_0) \Big), &\rho=0,
	\end{array}\right.
\end{eqnarray*}

\begin{eqnarray*}
	b_\sigma&=&\int_0^1 \frac{\partial^2 \ell}{\partial x\partial\sigma}\left(\theta_0,Q_{\gamma_0}(s)\right)H_{\gamma_0,\rho}\left(\frac{1}{-\log s} \right)ds\\
	&=& \left\{\begin{array}{ll}
		\frac{1}{\gamma_0^2\rho(\gamma_0+\rho)}\Big(-\gamma_0-\rho +(1+\gamma_0)(\gamma_0+2\rho)\Gamma(1+\gamma_0)-(1+\gamma_0)^2\rho\Gamma(1+2\gamma_0)&\\
	\quad+\gamma_0\Gamma(2-\rho)-\gamma_0(1+\gamma_0)(1-\rho)\Gamma(1+\gamma_0-\rho)\Big),&\rho<0,   \\
	\frac{1}{\gamma_0^3}\Big(-1+\gamma_0(\gamma_\ast-1)-(1+\gamma_0)^2\Gamma(1+2\gamma_0)+\Gamma(3+\gamma_0)+\gamma_0(1+\gamma_0)\Gamma'(1+\gamma_0) \Big), &\rho=0.
\end{array}\right.
\end{eqnarray*}
}

\section{Bounds for the derivatives of the likelihood}\label{app:C}

We  provide in this section upper bounds for the partial derivatives of the GEV log-likelihood,
\[
\ell(\theta,x)=-\left(1+\frac{1}{\gamma}\right)\log\left(1+\gamma\frac{x-\mu}{\sigma}\right)-\left(1+\gamma\frac{x-\mu}{\sigma}\right)^{-1/\gamma}-\log\sigma,
\] 
for $\theta=(\gamma,\mu,\sigma)$ and $x$ such that $1+\gamma\frac{x-\mu}{\sigma}>0$.
\begin{proposition}\label{prop3}
Let $\theta_0=(\gamma_0,0,1)$ with $\gamma_0\in\mathbb{R}$. For all $\varepsilon>0$, there exists a neighbourhood $N_0$ of $\theta_0$ and a constant $C>0$ such that, for all $\theta\in N_0$ and $x$ such that $1+\gamma(x-\mu)/\sigma>0$, we have
\begin{eqnarray*}
\left|\frac{\partial \ell}{\partial x^{i}\partial \theta^{1-i}} \right|&\leq& C\max\left(z^{i\gamma_0-\varepsilon},z^{1+i\gamma_0+\varepsilon},z^{\gamma_0-\varepsilon},z^{1+\gamma_0+\varepsilon} \right),\quad i=0,1,\\
\left|\frac{\partial^2 \ell}{\partial x^{i}\partial \theta^{2-i}} \right|&\leq& C\max\left(z^{i\gamma_0-\varepsilon},z^{1+i\gamma_0+\varepsilon},z^{2\gamma_0-\varepsilon},z^{1+2\gamma_0+\varepsilon} \right),\quad i=0,1,2,\\
\left|\frac{\partial^3 \ell}{\partial x^{i}\partial \theta^{3-i}} \right|&\leq& C\max\left(z^{i\gamma_0-\varepsilon},z^{1+i\gamma_0+\varepsilon},z^{3\gamma_0-\varepsilon},z^{1+3\gamma_0+\varepsilon} \right),\quad i=0,1,2,3,
\end{eqnarray*}
where 
\[
z=z(\theta,x)=\left(1+\gamma\frac{x-\mu}{\sigma}\right)^{-1/\gamma}>0,\quad \text{for } 1+\gamma\frac{x-\mu}{\sigma}>0.
\]
The notation $\partial\theta$ denotes either $\partial\gamma$, $\partial\sigma$ or $\partial\mu$ and, similarly for higher order derivatives, $\partial\theta^2$ denotes $\partial\gamma^2$, $\partial\gamma\partial\mu$, $\partial\gamma\partial\sigma$, $\partial\mu^2$ \dots
\end{proposition}
Note that the constant $C>0$ appearing in Proposition~\ref{prop3} and in the proofs below may change from line to line. Proposition~\ref{prop3} gathers with short notations several different inequalities. For instance, the first inequality with $i=1$ yields 
\[
\left|\frac{\partial \ell}{\partial x} \right|
\leq C\max\left(z^{\gamma_0-\varepsilon},z^{1+\gamma_0+\varepsilon} \right),
\]
while the third inequality with $i=1$ yields
\begin{align*}
\left|\frac{\partial^3 \ell}{\partial x\partial\gamma\partial\sigma} \right|
& \leq C\max\left(z^{\gamma_0-\varepsilon},z^{1+\gamma_0+\varepsilon},z^{3\gamma_0-\varepsilon},z^{1+3\gamma_0+\varepsilon} \right)\\
&\leq\left\{ \begin{array}{lll} 
C\max\left(z^{\gamma_0-\varepsilon},z^{1+3\gamma_0+\varepsilon} \right) &\mbox{if}& \gamma_0\geq 0\\
C\max\left(z^{3\gamma_0-\varepsilon},z^{1+\gamma_0+\varepsilon} \right) &\mbox{if}& \gamma_0\leq 0.
\end{array}\right.
\end{align*}

\medskip
For $\theta=(\gamma,\mu,\sigma)$, we have
\begin{equation}\label{eq:gandl}
\ell(\theta,x)=g\left(\gamma,\frac{x-\mu}{\sigma}\right)-\log\sigma,\quad g(\gamma,x)=-\left(1+\frac{1}{\gamma}\right)\log(1+\gamma x)-(1+\gamma x)^{-1/\gamma}
\end{equation}
where the function $g$ 
is the log-likelihood of the one parameter GEV distribution with $\gamma\in\mathbb{R}$  (i.e. $\mu=0$ and $\sigma=1$). With the notation
\begin{equation}\label{eq:def_z}
 z(\gamma,x)=(1+\gamma x)^{-1/\gamma},\quad 1+\gamma x>0
\end{equation}
we have,
\begin{equation} \label{eq:gand2}
g(\gamma,x)= (1+\gamma)\log z(\gamma,x)-z(\gamma,x).
\end{equation}

The proof of Proposition~\ref{prop3} relies on three lemmas providing upper bounds for the derivatives of $z(\gamma,x)$ and $g(\gamma,x)$. Define $h:\mathbb{R}\to\mathbb{R}$ by 
\begin{equation}\label{eq:def_h}
h(x)=\left\{\begin{array}{cc} (e^x-1-x)/x^2 & \mbox{for } x\neq 0\\  1/2 &\mbox{for } x= 0 \end{array}\right.
\end{equation}
and denote by $h^{(n)}(x)$ its derivative of order $n=0,1,2$.
\begin{lemma}\label{lem:h}
The function $h:\mathbb{R}\to\mathbb{R}$ in \eqref{eq:def_h} 
is twice continuously differentiable and $h^{(n)}(x)=O(\max(1,e^x))$ for all $x\in\mathbb{R}$, and $n=0,1,2$. 
\end{lemma}
\begin{proof}[Proof of Lemma~\ref{lem:h}]
The function $h$ can be represented as the power series $h(x)=\sum_{n\geq 0}x^{n}/(n+2)!$ and is hence indefinitely continuously differentiable. From the asymptotic behaviour
\[
h(x)\sim  -\frac{1}{x} \mbox{ as }  x\to -\infty\quad ,\quad h(x)\sim  \frac{e^x}{x^2}\mbox{ as } x\to +\infty,
\]
we deduce that the function $x\mapsto |h(x)|/\max(1,e^x)$ is bounded on $\mathbb{R}$ since it is continuous with vanishing limits at $\pm\infty$. This proves the existence of $C>0$ such that $|h(x)|\leq C\max(1,e^x)$ for all $x\in\mathbb{R}$. The upper bound for the second and third derivatives is proved similarly since simple computations show that
\[
h'(x)\sim \frac{1}{x^2} \mbox{ as }  x\to -\infty\quad ,\quad h'(x)\sim  \frac{e^x}{x^2}\mbox{ as } x\to +\infty
\]
and
\[
h''(x)\sim -\frac{2}{x^3} \mbox{ as } x\to -\infty\quad ,\quad h''(x)\sim  \frac{e^x}{x^2}\mbox{ as } x\to +\infty.
\]
\end{proof}

From Lemma \ref{lem:h} it follows
\begin{equation}\label{eq:bound_hz}
\left|h^{(n)}(\gamma\log z)\right|=O\left(\max(1,z^\gamma\right), \quad n=0,1,2.
\end{equation}
We also  use throughout the elementary bound $|\log z|^k=O\left(z^{\pm\delta}\right)$, for all $k\in\mathbb{N}$ and $\delta>0$.
\begin{lemma}\label{lem:z}
Consider the function $z(\gamma,x)$ defined by Equation~\eqref{eq:def_z}.
For all $\gamma_0\in\mathbb{R}$ and $\varepsilon>0$, there exists $C>0$ and $\delta>0$  such that, for all $\gamma\in(\gamma_0-\delta,\gamma_0+\delta)$ and $1+\gamma x>0$,
\begin{eqnarray*}
\left|\frac{\partial z}{\partial x^{i}\partial \gamma^{1-i}} \right|&\leq& C\max\left(z^{1+i\gamma_0\pm\varepsilon},z^{1+\gamma_0\pm\varepsilon} \right),\quad i=0,1,\\
\left|\frac{\partial^2 z}{\partial x^{i}\partial \gamma^{2-i}} \right|&\leq& C\max\left(z^{1+i\gamma_0\pm\varepsilon},z^{1+2\gamma_0\pm\varepsilon} \right),\quad i=0,1,2,\\
\left|\frac{\partial^3 z}{\partial x^{i}\partial \gamma^{3-i}}\right|&\leq& C\max\left(z^{1+i\gamma_0\pm\varepsilon},z^{1+3\gamma_0\pm\varepsilon} \right)\quad i=0,1,2,3,
\end{eqnarray*}
\end{lemma}
\begin{proof}[Proof of Lemma~\ref{lem:z}]
Recall the definition~\eqref{eq:def_h} of the function $h$. The first order partial derivatives of $z$ equal 
\[
\frac{\partial z}{\partial x}= -z^{1+\gamma} \quad \mbox{and}\quad
\frac{\partial z}{\partial \gamma}= z (\log z)^2 h(\gamma\log z).
\]
Assuming $\gamma\in(\gamma_0-\delta,\gamma_0+\delta)$, we deduce 
\begin{equation}\label{eq:diff_z_1.1}
\left|\frac{\partial z}{\partial x}\right|\leq  C\max(z^{1+\gamma_0\pm\delta})\text{ and }\left|\frac{\partial z}{\partial \gamma}\right|\leq  C\max(z^{1\pm\delta},z^{1+\gamma_0\pm 2\delta})
\end{equation}
from which the bounds of the first order partial derivatives of $z$ follow.
The second order partial derivatives of $z$ are given by
\begin{eqnarray*}
\frac{\partial^2 z}{\partial x^2}&=& -(1+\gamma)z^\gamma \frac{\partial z}{\partial x}\\
\frac{\partial^2 z}{\partial x\partial \gamma}&=& -(1+\gamma)z^\gamma \frac{\partial z}{\partial \gamma}-(\log z)z^{1+\gamma}\\
\frac{\partial^2 z}{\partial \gamma^2}&=&\left\{(\log z+2)(\log z) h(\gamma\log z)+\gamma(\log z)^2h'(\gamma\log z)\right\}\frac{\partial z}{\partial \gamma}+z(\log z)^3h'(\gamma \log z).
\end{eqnarray*}
Combined with the previous bounds and \eqref{eq:bound_hz}, we obtain, for $\gamma\in(\gamma_0-\delta,\gamma_0+\delta)$,
\begin{eqnarray*}
\left|\frac{\partial^2 z}{\partial x^2}\right|&\leq& C\max(z^{1+2\gamma_0\pm 2\delta})\\
\left|\frac{\partial^2 z}{\partial\gamma \partial x}\right|&\leq& C\max(z^{1+\gamma_0\pm 2\delta},z^{1+2\gamma_0\pm 2\delta})
\\
\left|\frac{\partial^2 z}{\partial \gamma^2}\right|&\leq&C \max(z^{1\pm 2\delta},z^{1+2\gamma_0\pm 4\delta})
\end{eqnarray*}
from which the bounds for the second order partial derivatives of $z$ follow.
The case of third order derivatives can be dealt with similarly and we omit the details.
\end{proof}

\begin{lemma}\label{lemg2}
Consider the function $g(\gamma,x)$ defined by Equation~\eqref{eq:gand2}. For all $\gamma_0\in\mathbb{R}$ and $\varepsilon>0$, there exists $C>0$ and $\delta>0$ such that, for all $\gamma\in (\gamma_0-\delta,\gamma_0+\delta)$ and $1+\gamma x>0$,
\begin{eqnarray*}
\left|\frac{\partial g}{\partial x^{i}\partial \gamma^{1-i}} \right|&\leq& C\max\left(z^{i\gamma_0-\varepsilon},z^{1+i\gamma_0+\varepsilon},z^{\gamma_0-\varepsilon},z^{1+\gamma_0+\varepsilon} \right),\quad i=0,1,\\
\left|\frac{\partial^2 g}{\partial x^{i}\partial \gamma^{2-i}} \right|&\leq& C\max\left(z^{i\gamma_0-\varepsilon},z^{1+i\gamma_0+\varepsilon},z^{2\gamma_0-\varepsilon},z^{1+2\gamma_0+\varepsilon} \right),\quad i=0,1,2,\\
\left|\frac{\partial^3 g}{\partial x^{i}\partial \gamma^{3-i}} \right|&\leq& C\max\left(z^{i\gamma_0-\varepsilon},z^{1+i\gamma_0+\varepsilon},z^{3\gamma_0-\varepsilon},z^{1+3\gamma_0+\varepsilon} \right),\quad i=0,1,2,3.
\end{eqnarray*}
\end{lemma}
\begin{proof}[Proof of Lemma~\ref{lemg2}]
Equation \eqref{eq:gand2} expressing $g(\gamma,x)$ in terms of $z(\gamma,x)$ entails
\begin{equation*}
\frac{\partial g}{\partial x}= \left(\frac{1+\gamma}{z}-1\right)\frac{\partial z}{\partial x}\quad,\quad
\frac{\partial g}{\partial \gamma}= \left(\frac{1+\gamma}{z}-1\right)\frac{\partial z}{\partial \gamma}+\log z.
\end{equation*}
Using the upper bound for the first derivatives of $z(\gamma,x)$ in Lemma~\ref{lem:z}, we get, for $\gamma\in(\gamma_0-\delta,\gamma_0+\delta)$,
\begin{equation*}
\left|\frac{\partial g}{\partial x}\right|\leq C\max(z^{\gamma_0\pm \varepsilon},z^{1+\gamma_0\pm \varepsilon}) \quad,\quad
\left|\frac{\partial g}{\partial \gamma}\right|\leq C\max(z^{\pm \varepsilon},z^{1\pm \varepsilon},z^{\gamma_0\pm \varepsilon},z^{1+\gamma_0\pm \varepsilon}) .
\end{equation*}
This proves the upper bounds for the first derivatives of $g(\gamma,x)$ given in Lemma~\ref{lemg2}. Note that we can handle the $\pm$ sign since we have $-\varepsilon<\varepsilon<1-\varepsilon<1+\varepsilon$, for small $\varepsilon>0$. When considering the maximum of the power functions, only the extreme exponents $-\varepsilon<1+\varepsilon$ matter. A similar argument holds for $\gamma_0-\varepsilon<\gamma_0+\varepsilon<1+\gamma_0-\varepsilon<1+\gamma_0+\varepsilon$.

Similarly for the second order derivatives of $g(\gamma,x)$, the upper bounds are derived from the similar upper bounds for the partial derivatives of $z(\gamma,x)$ in Lemma~\ref{lem:z} (with $\varepsilon$ replaced by $\varepsilon/2$) together with the formulas
\begin{eqnarray*}
\frac{\partial^2 g}{\partial x^2}&=& \left(\frac{1+\gamma}{z}-1\right)\frac{\partial^2 z}{\partial x^2}-\frac{1+\gamma}{z^2}\left(\frac{\partial z}{\partial x}\right)^2\\
\frac{\partial^2 g}{\partial x\partial \gamma}&=& \left(\frac{1+\gamma}{z}-1\right)\frac{\partial^2 z}{\partial x\partial \gamma}
+\frac{1}{z}\frac{\partial z}{\partial x}-\frac{1+\gamma}{z^2} \frac{\partial z}{\partial x}\frac{\partial z}{\partial \gamma}\\
\frac{\partial^2 g}{\partial \gamma^2}&=& \left(\frac{1+\gamma}{z}-1\right)\frac{\partial^2 z}{\partial \gamma^2}
+\frac{2}{z}\frac{\partial z}{\partial \gamma}-\frac{1+\gamma}{z^2}\left(\frac{\partial z}{\partial \gamma}\right)^2.
\end{eqnarray*}
Partial derivatives of order $3$ are dealt with similarly. 
\end{proof}

\begin{proof}[Proof of Proposition~\ref{prop3}]
Equation \eqref{eq:gandl} implies that the partial derivatives of $\ell$ are closely related to those of $g$. For the first derivatives, we have
\begin{multicols}{2}
\noindent
\begin{eqnarray*}
\frac{\partial\ell}{\partial x}(\theta,x)&=& \frac{1}{\sigma}\frac{\partial g}{\partial x}\left(\gamma,\frac{x-\mu}{\sigma}\right)\\
\frac{\partial\ell}{\partial \gamma}(\theta,x)&=& \frac{\partial g}{\partial \gamma}\left(\gamma,\frac{x-\mu}{\sigma}\right)
\end{eqnarray*}
\begin{eqnarray*}
\frac{\partial\ell}{\partial \mu}(\theta,x)&=& -\frac{1}{\sigma}\frac{\partial g}{\partial x}\left(\gamma,\frac{x-\mu}{\sigma}\right)\\
\frac{\partial\ell}{\partial \sigma}(\theta,x)&=& -\frac{x-\mu}{\sigma^2}\frac{\partial g}{\partial x}\left(\gamma,\frac{x-\mu}{\sigma}\right)-\frac{1}{\sigma}.
\end{eqnarray*}
\end{multicols}
\noindent
These equalities together with Lemma~\ref{lemg2} yield
\[
\left|\frac{\partial\ell}{\partial x}\right|= \frac{1}{\sigma}\left|\frac{\partial g}{\partial x}\right| \leq C\max\left(z^{\gamma_0-\varepsilon},z^{1+\gamma_0+\varepsilon}\right),
\]
which corresponds to the first inequality in Proposition~\ref{prop3} with $i=0$. We also obtain 
\[
\left|\frac{\partial\ell}{\partial \gamma}\right|= \left|\frac{\partial g}{\partial \gamma}\right| \leq C\max\left(z^{-\varepsilon},z^{1+\varepsilon},z^{\gamma_0-\varepsilon},z^{1+\gamma_0+\varepsilon}\right)
\quad,\quad
\left|\frac{\partial\ell}{\partial \mu}\right| = \frac{1}{\sigma}\left|\frac{\partial g}{\partial x}\right| \leq C\max\left(z^{\gamma_0-\varepsilon},z^{1+\gamma_0+\varepsilon}\right)
\]
which implies the inequality in Proposition~\ref{prop3} with $i=1$ and the derivatives taken with respect to $\gamma$ and $\mu$ respectively. The case of the derivative with respect to $\sigma$ is slightly more difficult: we use the inequality
\[
\left|\frac{x-\mu}{\sigma}\right|=\left|\frac{z^{-\gamma}-1}{\gamma}\right|\leq C\max(z^{-\gamma_0- \delta},z^{-\gamma_0+ \delta},1),\quad \gamma\in(\gamma_0-\delta,\gamma_0+\delta),
\]
which implies
\begin{eqnarray*}
\left|\frac{\partial\ell}{\partial \sigma}\right|&=&  
\frac{1}{\sigma}\left|\frac{z^{-\gamma}-1}{\gamma}\frac{\partial g}{\partial x}+1\right| 
\leq C\max(z^{-\gamma_0- \delta},z^{-\gamma_0+ \delta},1)\max\left(z^{\gamma_0-\varepsilon},z^{1+\gamma_0+\varepsilon}\right) \\
& \leq& C\max(z^{-\varepsilon'},z^{1+\varepsilon'},z^{\gamma_0-\varepsilon'},z^{1+\gamma_0+\varepsilon'})
\end{eqnarray*}
for sufficiently small $\varepsilon'$. For the first inequality, we use Lemma~\ref{lemg2}. 
This proves the first inequality in Proposition~\ref{prop3} with $i=1$ and the derivatives taken with respect to $\sigma$.

The case of second order derivatives is dealt similarly with the relations
\begin{multicols}{2}
\noindent
\begin{eqnarray*}
\frac{\partial^2\ell}{\partial x^2}&=& \frac{1}{\sigma^2}\frac{\partial^2 g}{\partial x^2}\\
\frac{\partial^2\ell}{\partial \gamma\partial x}&=& \frac{1}{\sigma}\frac{\partial^2 g}{\partial \gamma\partial x}\\
\frac{\partial^2\ell}{\partial \mu\partial x}&=& -\frac{1}{\sigma^2}\frac{\partial^2 g}{\partial x^2}\\
\frac{\partial^2\ell}{\partial \sigma\partial x}&=& -\frac{x-\mu}{\sigma^3}\frac{\partial^2 g}{\partial x^2}\\
\frac{\partial^2\ell}{\partial \gamma^2}&=& \frac{\partial^2 g}{\partial \gamma^2}\\
\end{eqnarray*}
\begin{eqnarray*}
\frac{\partial^2\ell}{\partial \mu^2}&=& \frac{1}{\sigma^2}\frac{\partial^2 g}{\partial x^2}\\
\frac{\partial^2\ell}{\partial \sigma^2}&=& 2\frac{x-\mu}{\sigma^3}\frac{\partial g}{\partial x}+\frac{(x-\mu)^2}{\sigma^4}\frac{\partial^2 g}{\partial x^2}+\frac{1}{\sigma^2}\\
\frac{\partial^2\ell}{\partial\gamma\partial \mu}&=& -\frac{1}{\sigma}\frac{\partial^2 g}{\partial \gamma\partial x}\\
\frac{\partial^2\ell}{\partial\gamma\partial \sigma}&=& -\frac{x-\mu}{\sigma^2}\frac{\partial^2 g}{\partial \gamma\partial x}\\
\frac{\partial^2\ell}{\partial \mu\partial\sigma}&=& \frac{1}{\sigma^2}\frac{\partial g}{\partial x}+\frac{x-\mu}{\sigma^3}\frac{\partial^2 g}{\partial x^2}.
\end{eqnarray*}
\end{multicols}
Using these relations, the second inequality in Proposition~\ref{prop3} follows from Lemma~\ref{lemg2}. Checking all the different cases is relatively tedious but elementary. Details are omitted. Similar formulas hold for derivatives of order $3$ and the resulting bounds have been checked with Mathematica$^\circledR$.
\end{proof}

\section*{Acknowledgements} 
The authors would like to thank Laurens de Haan for his useful suggestions and support. 

Research partially funded by VolkswagenStiftung Support for Europe -- WEX-MOP and, FCT - Fundac\~ao para a Ci\^encia e a Tecnologia, Portugal, through UID/MAT/00006/2013 and UID/Multi/04621/2013.


\begin{thebibliography}{10}

\bibitem{BdH74}
A.A. Balkema and L.~de~Haan.
\newblock Residual life time at great age.
\newblock {\em Ann. Probability}, 2:792--804, 1974.

\bibitem{BGST04}
J.~Beirlant, Y.~Goegebeur, J.~Teugels, and J.~Segers.
\newblock {\em Statistics of extremes:  Theory and applications}.
\newblock Wiley Series in Probability and Statistics. John Wiley \& Sons Ltd.,
  Chichester, 2004.

\bibitem{BS16c}
A.~B{\"u}cher and J.~Segers.
\newblock Maximum likelihood estimation for the fréchet distribution based on
  block maxima extracted from a time series.
\newblock Preprint arXiv:1511.07613.

\bibitem{BS16b}
A.~B{\"u}cher and J.~Segers.
\newblock On the maximum likelihood estimator for the generalized extreme-value
  distribution.
\newblock Preprint arXiv:1601.05702.

\bibitem{C01}
S.~Coles.
\newblock {\em An Introduction to Statistical Modeling of Extreme Values}.
\newblock Springer-Verlag, London, 2001.

\bibitem{dHF06}
L.~de~Haan and A.~Ferreira.
\newblock {\em Extreme value theory: An introduction}.
\newblock Springer Series in Operations Research and Financial Engineering.
  Springer, New York, 2006.

\bibitem{D15}
C.~Dombry.
\newblock Existence and consistency of the maximum likelihood estimators for
  the extreme value index within the block maxima framework.
\newblock {\em Bernoulli}, 21(1):420--436, 2015.

\bibitem{DdHL03}
H.~Drees, L.~de~Haan, and D.~de~Li.
\newblock On large deviation for extremes.
\newblock {\em Stat. Probab. Letters}, 1(64):51--62, 2003.

\bibitem{EKM97}
P.~Embrechts, C.~Kl{\"u}ppelberg, and T.~Mikosch.
\newblock {\em Modelling extremal events}, volume~33 of {\em Applications of
  Mathematics (New York)}.
\newblock Springer-Verlag, Berlin, 1997.

\bibitem{FdH15}
A.~Ferreira and L.~de~Haan.
\newblock On the block maxima method in extreme value theory: {PWM} estimators.
\newblock {\em Ann. Statist.}, 43(1):276--298, 2015.

\bibitem{G58}
E.J. Gumbel.
\newblock {\em Statistics of extremes}.
\newblock Columbia University Press, New York, 1958.

\bibitem{HW87}
J.R.M. Hosking and J.R. Wallis.
\newblock Parameter and quantile estimation for the generalized {P}areto
  distribution.
\newblock {\em Technometrics}, 29(3):339--349, 1987.

\bibitem{HWW85}
J.R.M. Hosking, J.R. Wallis, and E.F. Wood.
\newblock Estimation of the generalized extreme-value distribution by the
  method of probability-weighted moments.
\newblock {\em Technometrics}, 27(3):251--261, 1985.

\bibitem{P75}
J.~Pickands, III.
\newblock Statistical inference using extreme order statistics.
\newblock {\em Ann. Statist.}, 3:119--131, 1975.

\bibitem{PW80}
P.~Prescott and A.T. Walden.
\newblock Maximum likelihood estimation of the parameters of the generalized
  extreme-value distribution.
\newblock {\em Biometrika}, 67(3):723--724, 1980.

\bibitem{SW86}
G.R. Shorack and J.A. Wellner.
\newblock {\em Empirical processes with applications to statistics}.
\newblock Wiley Series in Probability and Mathematical Statistics: Probability
  and Mathematical Statistics. John Wiley \& Sons Inc., New York, 1986.

\bibitem{S85}
R.L. Smith.
\newblock Maximum likelihood estimation in a class of nonregular cases.
\newblock {\em Biometrika}, 72(1):67--90, 1985.

\bibitem{vdV98}
A.W. van~der Vaart.
\newblock {\em Asymptotic statistics}, volume~3 of {\em Cambridge Series in
  Statistical and Probabilistic Mathematics}.
\newblock Cambridge University Press, Cambridge, 1998.

\end{thebibliography}
\end{document}